\documentclass[12pt]{amsart}
\usepackage{amssymb}
\usepackage{amsmath,amscd}
\usepackage{amsfonts,latexsym,verbatim,amscd,mathrsfs,color,array}
\usepackage[all,cmtip]{xy}
\usepackage{mathrsfs}
\usepackage{multirow}
\usepackage{graphicx}
\usepackage{amsfonts, amsthm}
\usepackage{bbm}
\usepackage[hmargin=1in,vmargin=1in]{geometry}
\usepackage{mathdots}
\usepackage{stmaryrd}
\usepackage{MnSymbol}
\usepackage{bbm}
\usepackage[hidelinks]{hyperref}

\allowdisplaybreaks

\def\Xint#1{\mathchoice
	{\XXint\displaystyle\textstyle{#1}}%
	{\XXint\textstyle\scriptstyle{#1}}%
	{\XXint\scriptstyle\scriptscriptstyle{#1}}%
	{\XXint\scriptscriptstyle\scriptscriptstyle{#1}}%
	\!\int}
\def\XXint#1#2#3{{\setbox0=\hbox{$#1{#2#3}{\int}$ }
		\vcenter{\hbox{$#2#3$ }}\kern-.6\wd0}}

\def\dashint{\Xint-}

\newcommand{\cev}[1]{\reflectbox{\ensuremath{\vec{\reflectbox{\ensuremath{#1}}}}}}

\usepackage{accents}
\newlength{\dhatheight}

\renewcommand{\Re}{\operatorname{Re}}
\renewcommand{\Im}{\operatorname{Im}}
\newcommand{\osc}{\operatorname{osc}}
\DeclareMathSymbol{\intprod}{\mathbin}{MnSyC}{'270}
\newcommand{\LB}{\left[}
\newcommand{\RB}{\right]}
\newcommand{\LA}{\left\langle}
\newcommand{\RA}{\right\rangle}

\newcommand{\Z}{{\mathbb Z}}
\newcommand{\N}{{\mathbb N}}
\newcommand{\C}{{\mathbb C}}
\renewcommand{\P}{{\mathbb P}}

\newcommand{\R}{{\mathbb R}}

\newcommand{\SO}{{\mathrm{SO} }}

\newcommand{\eps}{{\varepsilon}}

\newcommand{\T}{{\mathcal T}}

\newcommand{\dist}{{\mathrm{dist}}}
\newcommand{\tr}{{\operatorname{tr}}}

\newcommand{\dbar}{{\bar{\partial}}}
\newcommand{\p}{{\partial}}

\newtheorem{thm}{Theorem}[section]
\newtheorem{lemma}[thm]{Lemma}
\newtheorem*{lemma*}{Lemma}
\newtheorem{prop}[thm]{Proposition}

\newtheorem{cor}[thm]{Corollary}

\newtheorem*{conj*}{Conjecture}

\newenvironment{claim}{\par\medskip\noindent\textit{Claim}\space}{\par\medskip}
\newenvironment{claimproof}{\par\noindent\textit{Proof of claim.}\space}{\hfill$\diamond$\medskip\par}

   \newtheoremstyle{others}
     {3pt}
     {2pt}
     {}
     {}
     {\bf}
     {.}
     {.5em}
     {}

\theoremstyle{others}
\newtheorem{rmk}[thm]{Remark}
\newtheorem*{rmk*}{Remark}
\newtheorem{defn}[thm]{Definition}

\numberwithin{equation}{section}

\begin{document}

\title{\L ojasiewicz inequalities for maps of the 2-sphere}
\author{Alex Waldron}
\address{University of Wisconsin, Madison}
\email{waldron@math.wisc.edu}

\begin{abstract}
We prove a \emph{\L ojasiewicz-Simon inequality}
\begin{equation*}
\left| E(u) - 4\pi n \right| \leq C \| \T(u) \|_{L^2}^\alpha
\end{equation*}
for maps $u \in W^{2,2}\left( S^2, S^2 \right).$ 
The inequality holds with $\alpha = 1$ in general and with $\alpha > 1$ 
unless $u$ is nearly constant on an open set. We obtain polynomial convergence of weak solutions of harmonic map flow $u(t) : S^2 \to S^2$ as $t \to \infty$ on compact domains away from the singular set, assuming that the body map is nonconstant. 
The proof uses Topping's repulsion estimates together with polynomial lower bounds on the energy density 
coming from a bubble-tree induction argument.
\end{abstract}

\maketitle

\thispagestyle{empty}

\tableofcontents

\section{Introduction}

\subsection{Background and main result}
Let $(M, g)$ and $(N,h)$ be compact Riemannian manifolds. Recall that a sufficiently regular map $v : M \to N$ is called \emph{harmonic} if its tension field
\begin{equation}\label{tensionfield}
\T(v) = \tr_g \nabla dv
\end{equation}
vanishes identically. Equivalently, $v$ is a critical point of the Dirichlet energy functional
\begin{equation*}
E(u) = \frac12 \int_M |du|^2 \, dV.
\end{equation*}
We are chiefly concerned with the fundamental case $M = N = S^2,$ although some of our results are stated more generally. 

A classical result due to Lemaire and Wood \cite[(11.5)]{eellslemairereport} states that any harmonic map $v : S^2 \to S^2$ must either be holomorphic or antiholomorphic; hence
\begin{equation}\label{4pideg}
E(v) = 4 \pi | \deg(v) |
\end{equation}
and $v$ attains the minimum allowable energy within its homotopy class. The proof relies on a famous trick: 
since the Hopf differential (a certain quadratic expression in the components of $du$) of a harmonic map is holomorphic, it necessarily vanishes, forcing the map $v$ to be holomorphic or antiholomorphic. 

It is natural to ask whether a quantitative version of this result holds true; specifically, whether the energy of a map with small tension must be close to a predetermined value. In his thesis \cite{toppingrigidity}, Topping obtained a remarkable first result in this direction: any map $u : S^2 \to S^2$ with
\begin{equation}\label{Toppingslowenergy}
E(u) < 4 \pi |\deg(u)| + \delta_0,
\end{equation}
must obey the estimate 
\begin{equation}\label{Toppingsfirstestimate}
E(u) \leq  4 \pi |\deg(u) | + C \| \T(u) \|^2,
\end{equation}
where $\|\cdot \| = \| \cdot \|_{L^2}.$
For initial data satisfying (\ref{Toppingslowenergy}), Topping used (\ref{Toppingsfirstestimate}) to prove that 
Struwe's weak solution of the harmonic map flow converges exponentially away from a finite set of singular points.

Without the assumption (\ref{Toppingslowenergy}), one can no longer expect 
a map with low tension to be nearly holomorphic; instead, one might expect such a map to be close to a \emph{sum} of holomorphic and anti-holomorphic maps. This intuition was borne out by the work of Ding-Tian \cite{dingtian}, Qing \cite{qinghmfsphere}, and Wang \cite{wangchangyoububblephenomena}, establishing that a sequence of maps $u_i$ 
with $\| \T(u_i) \| \to 0$ must subconverge to a \emph{bubble tree} of harmonic maps. The reader may see Theorem \ref{thm:bubbletree} of the present paper for a detailed statement of this theory, or recent work of Jendrej-Lawrie-Schlag \cite{jendrejlawrieschlag} for more precise results in the parabolic setting.

In fact, the energies $E(u_i)$ converge to the sum of the energies of the maps in the tree, which by (\ref{4pideg}), is an integer multiple of $4 \pi:$
\begin{equation*}
E(u_i) \to 4 \pi n.
\end{equation*}
Another remarkable result of Topping \cite{toppingannals} generalizes the estimate (\ref{Toppingsfirstestimate}) to this setting: under certain assumptions on the nature of the bubble tree (described below), 
we have
\begin{equation}\label{Toppingssecondestimate}
| E(u_i) - 4 \pi n | \leq C' \| \T(u_i)\|^2.
\end{equation}
Topping was able to use (\ref{Toppingssecondestimate}), which he calls a ``quantization estimate,'' to again obtain exponential convergence of harmonic map flow on compact domains away from the singular set. 

The mechanism of Topping's second theorem is more complex than the first, incorporating a ``repulsion'' effect between the body map, assumed to be holomorphic, and any antiholomorphic bubbles---in other words, a quantification of the Hopf-differential trick.
To ensure a sufficiently strong repulsion effect, Topping makes the following assumptions: (1) holomorphic and antiholomorphic bubbles do not occur at the same points, and (2) the energy density of the body map does not vanish at any of the antiholomorphic bubble points---see \cite{toppingannals} for precise statements. As shown later \cite{toppingzeroenergy}, these hypotheses are necessary for an estimate of the form (\ref{Toppingssecondestimate}) to hold.

However, if 
one settles for a weaker estimate in place of (\ref{Toppingssecondestimate}), namely, a \emph{\L ojasiewicz(-Simon) inequality}
\begin{equation}\label{introloj}
| E(u) - 4 \pi n | \leq C \| \T(u) \|^\alpha,
\end{equation}
then it turns out that Topping's assumptions can be removed. 
We shall prove (\ref{introloj}) with $\alpha = 1$ in full generality and with $\alpha > 1$ assuming only that the body map is not identically constant. For a solution of harmonic map flow $u(t) : S^2 \to S^2,$ this allows us to conclude the following: if the weak limit along a given sequence of times tending to infinity is nonconstant, then $u(t)$ must converge polynomially 
on compact domains away from the singular set. In particular, weak subsequential limits are unique in this case. 

\subsection{Related work} Our main theorem can be compared with \L ojasiewicz inequalities and uniqueness-of-subsequential-limit results that have appeared in similar contexts.

Assuming that $u$ is $C^{2,\alpha}$-close to a fixed harmonic map, (\ref{introloj}) follows directly from fundamental work of Leon Simon \cite{leonsimon}. 
Historically, extensions of Simon's theorem to singular contexts have been exceedingly rare, Topping's theorems \cite{toppingrigidity, toppingannals} being the best and perhaps earliest examples. Another instance is the work of Daskalopoulos-Wentworth \cite{daskwentconvergence, daskwentblowup} (generalized by Sibley \cite{sibleythesiscrelle} and Sibley-Wentworth \cite{sibleywentworth}) establishing convergence modulo gauge and uniqueness of the bubbling set for Hermitian Yang-Mills flow on holomorphic bundles over compact K\"ahler manifolds. 
In his thesis \cite[Theorem 3.7-Corollary 3.8]{instantons}, the author also gave an infinite-time exponential convergence result for 4D Yang-Mills flow analogous to Topping's first convergence theorem, although in that context the bubbling set turned out to be empty.

More recently, Malchiodi, Rupflin, and Sharp \cite{malchiodirupflinsharplojasiewicz} obtained a \L ojasiewicz inequality for the $H$-functional that applies near ``simple bubble trees,'' consisting of a single bubble on a constant body map (assuming that the domain is a surface of genus at least one). Rupflin \cite{rupflinhmlojasiewicz} adapted this technique to the Dirichlet energy of harmonic maps and succeeded to work with simple bubble trees mapping into a general real-analytic target. 
Meanwhile, Rupflin \cite{rupflinlowenergylevels} has also obtained \L ojasiewicz inequalities for more general low-energy maps under realistic assumptions, including 
\cite[Corollary 1.6]{rupflinlowenergylevels}, which is a direct precedent of our estimate (\ref{introloj}).
(For recent progress on the related question of closeness of $u$ to a genuine holomorphic map under the hypothesis (\ref{Toppingslowenergy}), see Bernand-Mantel, Muratov, and Simon \cite{bernandmantelrigidity}, Topping \cite{toppingdegreeonerigidity}, and Rupflin \cite{rupflinquantitativerigidity}.)

A related trend initiated by Colding and Minicozzi \cite{coldingminicozzigenericI} aims to develop \L ojasiewicz-type estimates suitable for proving uniqueness of blowup limits at ``generic'' singularities. These estimates are designed for analyzing type-I singularities of the mean curvature flow and have been remarkably successful on that front \cite{coldingminicozziuniquenesslojasiewicz, coldingminicozzisingularset, chodoshchoischulzegenericmcf, chodoshchoischulzegenericmcfII, chodoshmantoulidisschulzegenericlowentropyII}. Lotay, Schulze, and Szekelyhidi \cite{lotayschulzeszekelyhidilagrangianmcfuniqueness} have also obtained a strong uniqueness result for certain singularities of Lagrangian mean curvature flow, which are of type II. 
As we shall demonstrate in forthcoming work, a version of the inequality (\ref{introloj}) can also be used to analyze (type-II) finite-time singularities of 2D harmonic map flow. 

Finally, we mention the recent \L ojasiewicz inequality of Deruelle-Ozuch \cite{deruelleozuchlojasiewicz} for a version of Perelman's $\lambda$-functional near a Ricci-flat ALE space, which has applications to infinite-time convergence of Ricci-DeTurck flow \cite{deruelleozuchdynamicalinstability}.

\subsection{\L ojasiewicz inequalities} 
The simplest version of our main result is as follows.

\begin{thm}\label{thm:loj}  


\noindent (a) Let $k \in \N.$ 
Given any map $u \in W^{2,2}( S^2, S^2)$ with $E(u) \leq 4 \pi k,$ there exists an integer $n \in \{ 0, \ldots, k\}$ such that 
\begin{equation*}\label{loj:a}
\left| E(u) - 4 \pi n \right| \leq C_{\ref{thm:loj}a} \|\T(u)\|.
\end{equation*}
Here $\|\cdot \| = \|\cdot\|_{L^2(S^2)}$ and $C_{\ref{thm:loj}a}$ depends only on $k.$

\vspace{2mm}

\noindent (b) 
There exist $L \in \N_0$ and $\eps_0 > 0,$ depending only on $k,$ as follows; fix any $1 < \alpha < \frac{2L + 2}{2L + 1},$ $0 < \kappa < \eps_0,$ and open sets $\Gamma \Subset \hat{\Gamma} \subset S^2.$ Let $u \in W^{2,2}( S^2, S^2)$ be any map with $E(u) \leq 4 \pi k$ as well as 
\begin{equation}\label{loj:Gammahatassumption}
E(u, \hat{\Gamma}) < \eps_0
\end{equation}
and
\begin{equation}\label{loj:Gammaassumption}
E(u, \Gamma) \geq \kappa.
\end{equation}
Then $u$ obeys the \L ojasiewicz inequality
\begin{equation*}\label{loj:bestimate}
\left| E(u) - 4 \pi n \right| \leq C_{\ref{thm:loj}b} \| \T(u) \|^{\alpha }.
\end{equation*}
Here $C_{\ref{thm:loj}b}$ depends on $k, \kappa, \alpha,$ and the geometry of $\Gamma$ and $\hat{\Gamma}.$

\end{thm}

We have also the following more quantitative versions. Given $x \in S^2$ and $0 < r < \infty,$ let $D_r(x)$ denote the image in $S^2$ of the disk of radius $r$ in the stereographic coordinate chart centered at $x,$ i.e.
\begin{equation*}
D_r(x) = B_{2 \arctan(r)}(x).
\end{equation*}
We shall also write
\begin{equation*}
U^{\rho_2}_{\rho_1}(x) = D_{\rho_2}(x) \setminus \overline{ D_{\rho_1}(x) }.
\end{equation*}
Fix $\ell \in \{0, \ldots, k\}.$
Let $x_i \in S^2$ and $0 < \lambda_i \leq \frac{1}{\sqrt{\ell} },$ for $i = 1, \ldots, \ell,$ be such that $D_{2\lambda_i}(x_i) \cap D_{2 \lambda_j} (x_j) = \emptyset$ for $i \neq j,$ and write
$$\Omega = S^2 \setminus \cup_i D_{\lambda_i}(x_i), \qquad U_i = U_{\lambda_i}^{2\lambda_i}(x_i) , \qquad \hat{U}_i = U_{\lambda_i/2}^{4\lambda_i}(x_i).$$
The next result is a refined version of Theorem \ref{thm:loj}$a.$

\begin{thm}\label{thm:quantitativelambdaloj} Fix $\beta, \kappa > 0.$ Suppose that $u \in W^{2,2}(S^2, S^2)$ satisfies $E(u) \leq 4\pi k$ as well as
\begin{equation}\label{quantitativeloj:edbarsmall}
\sum_{i = 1}^\ell E \left( u, \hat{U}_i \right) +  E_{\dbar}\left(u, \Omega \right) < \eps_0,
\end{equation}
where $\eps_0$ depends only on $k,$ and
\begin{equation}\label{quantitativeloj:aassumption}
E_{\p}\left( u, U_i \right) \geq \kappa E_{\dbar} \left(u, U_i \right) 
\end{equation}
for $i = 1, \ldots, \ell.$ 
Then there exists $n \in \{0, \ldots, k\}$ such that
\begin{equation*}
 \left| E(u) - 4 \pi n \right| \leq C_{\ref{thm:quantitativelambdaloj}} \left( \| \T(u) \|^2 + \lambda^{1 - \beta} \| \T(u) \| \right).
\end{equation*}
Here $\lambda = \max \lambda_i$ and $C_{\ref{thm:quantitativelambdaloj}}$ depends on $k, \kappa,$ and $\beta.$

A similar result holds after reversing the roles of $E_\p$ and $E_\dbar.$
\end{thm}

The following is a local version of the last theorem.

\begin{cor}\label{cor:localonebubbleloj}
Let $B_{2} \subset \R^2$ denote a ball of radius two and let $0 \leq \lambda \leq 1.$ Suppose that $u \in W^{2,2}(B_2, S^2)$ satisfies $E(u) \leq 4\pi k$ as well as
\begin{equation}
E(u, B _2 \setminus B_{\lambda}) < \eps_0,
\end{equation}
where $\eps_0$ depends only on $k.$
Then there exists $n \in \{0, \ldots, k\}$ such that
\begin{equation*}
 \left| E(u) - 4 \pi n \right| \leq C_{\ref{cor:localonebubbleloj} } \left( \| \T(u) \|_{L^2(B_2)}^2 + E(u, B_{2} \setminus B_1) + \lambda^{1 - \beta} \left( \| \T(u) \|_{L^2(B_2)} + \sqrt{E(u, B_{2} \setminus B_1)} \right) \right).
\end{equation*}
Here $C_{\ref{cor:localonebubbleloj}}$ depends only on $k$ and $\beta.$
\end{cor}

The next result is a version of Theorem \ref{thm:loj}$b$ in which the domains $\cup U_i \Subset \cup \hat{U}_i$ play the role of $\Gamma \Subset \hat{\Gamma}.$ 

\begin{thm}\label{thm:quantitativealphaloj} 
Let $1 < \alpha < \frac{2L + 2}{2L + 1}$ and $0 < \kappa < \eps_0,$ as in Theorem \ref{thm:loj}$b.$
Suppose that $u \in W^{2,2}(S^2, S^2)$ satisfies $E(u) \leq 4\pi k,$  (\ref{quantitativeloj:edbarsmall}), and
\begin{equation}\label{quantitativealphaloj:assumption}
E_{\p} \left( u, U_i \right) \geq \kappa
\end{equation}
for $i = 1,\ldots, \ell.$
Then $u$ obeys the \L ojasiewicz inequality
\begin{equation*}\label{quantitativealphaloj:estimate}
\left| E(u) - 4 \pi n \right| \leq C_{\ref{thm:quantitativealphaloj}} \| \T(u) \|^{\alpha }.
\end{equation*}
Here $C_{\ref{thm:quantitativealphaloj}}$ depends on $k, \kappa,$ and $\alpha.$

A similar result holds after reversing the roles of $E_\p$ and $E_\dbar.$
\end{thm}

Theorems \ref{thm:loj}-\ref{thm:quantitativealphaloj} will all follow by contradiction from Theorem \ref{thm:sequenceloj} below, which states the same estimates in the context of sequences of maps with tension fields tending to zero in $L^2$ (so-called ``almost-harmonic sequences'').



\vspace{2mm}

\subsection{Applications to harmonic map flow} We can now state our harmonic-map-flow results.

\begin{thm}\label{thm:flowconv} 
Let $k, \ell, \alpha, \kappa,$ $\{x_i\}_{i = 1}^\ell,$ and $\{\lambda_i\}_{i = 1}^\ell$ be as above. There exists $\delta_0 > 0,$ depending only on $k, \alpha,\kappa,$ and $\min \{\lambda_i\},$ as follows. 

Let $u_0 \in W^{1,2}(S^2, S^2)$ be a map 
satisfying the assumptions of Theorem \ref{thm:quantitativealphaloj}, 
and suppose further that
$$\dist \left( E(u_0) , 4\pi \Z \right) \leq \delta$$
for some $0 < \delta \leq \delta_0.$
Let $u(t)$ be the Struwe solution of harmonic map flow with $u(0) = u_0,$ and let $0 < T \leq \infty$ be the maximal time such that
\begin{equation}\label{mainthm:Tdef}
\dist \left( E(u(t) ) , 4\pi \Z \right) < \delta
\end{equation}
for $0 < t < T.$

\vspace{2mm}

\noindent (a) We have
\begin{equation}
\dist \left( E(u(t)) , 4\pi \Z \right) \leq \left( \delta^{ \frac{\alpha - 2}{\alpha} } + \frac{2 - \alpha}{\alpha}  \min \LB t, T - t \RB \right)^{\frac{\alpha}{\alpha - 2}},
\end{equation}
\begin{equation}\int_{t}^{T-t} \| \T(u(s)) \|_{L^2(S^2)} \, ds \leq \frac{\alpha}{\alpha - 1} \left( \delta^{ \frac{\alpha - 2}{\alpha} } + \frac{2 - \alpha}{\alpha}  t \right)^{\frac{\alpha - 1}{\alpha - 2}},
\end{equation}
and
\begin{equation}
\| u(t) - u_0 \|_{L^2(S^2)} \leq \frac{\alpha}{\alpha - 1} \delta^{\frac{\alpha - 1}{\alpha}}
\end{equation}
for $0 \leq t < T.$

\vspace{2mm}

\noindent (b) If $T = \infty,$ then there exists a nonconstant holomorphic map $u_\infty: S^2 \to S^2$ such that
\begin{equation}\label{L2convergence}
\| u(t) - u_\infty \|_{L^2(S^2)} \leq \frac{\alpha}{\alpha - 1} \left( \delta^{\frac{\alpha - 2}{\alpha} } + \frac{2 - \alpha}{\alpha}  t \right)^{\frac{\alpha - 1}{\alpha - 2}}. 
\end{equation}
Moreover, there exists a finite set of points $z_j \in S^2$ such that given any domain $\Gamma \Subset S^2 \setminus \{z_j\}$ and $m \in \N,$ there exists $D > 0$ such that
\begin{equation}\label{Ckconvergence}
\| u(t) - u_\infty \|_{C^m (\Gamma)} < D t^{\frac{ \alpha - 1}{\alpha - 2}}
\end{equation}
for $t$ sufficiently large.

\end{thm}

\begin{cor}\label{cor:flowconv} Let $t_i$ be any sequence of times tending to infinity. Given a Struwe solution $u(t) : S^2 \to S^2$ of harmonic map flow with $E(u(0)) \leq 4 \pi k,$ pass to any subsequence such that $u(t_i) \rightharpoonup u_\infty$ weakly in $W^{1,2}.$ 
If $u_\infty$ is nonconstant, then $u(t) \to u_\infty$ polynomially in $L^2(S^2),$ in the sense of (\ref{L2convergence}), and in $C^\infty$ on compact domains away from the bubbling set, in the sense of (\ref{Ckconvergence}).
\end{cor}

\vspace{2mm}

\subsection{Idea of proof} We follow a strategy of decomposing an almost-harmonic map into an alternating sum of almost-holomorphic and almost-antiholomorphic maps. The basic estimate of Topping's first theorem can be applied on each map, modulo a boundary term; the goal is to make sure that these boundary terms are taken on annuli where the holomorphic and antiholomorphic energies are comparable and bounded below by a power of the radius. Our analysis starts from the observation that, in this case, the boundary term can be controlled using Topping's repulsion estimate, giving precisely a (pre-)\L ojasiewicz inequality. 

The main difficulty in exploiting this observation
is to obtain the required lower (and upper) bounds on the energy densities 
in the neck regions.
These ultimately follow from a basic three-annulus estimate and its generalization to a ``multi-annulus estimate,'' needed to pass the lower bounds across ghost bubbles. 
We assemble the estimates using an inward/outward induction argument over the bubble tree.

\vspace{2mm}

\subsection{Acknowledgement} The author was partially suppored by NSF DMS-2004661 during the preparation of this article.

\vspace{10mm}

\section{Preliminaries}

In \S \ref{ss:mapstoKahler}, we recall the K\"ahler formalism for harmonic maps. In \S \ref{ss:sphere}, we recall the explicit formulae for the quantities defined in \S \ref{ss:mapstoKahler} in the case $\Sigma = N = S^2,$ viewed as the unit sphere inside $\R^3.$ 
For the purposes of understanding our main results, the reader may feel free to skip \S \ref{ss:mapstoKahler} and refer only to the formulae in \S \ref{ss:sphere}.

\subsection{Harmonic maps to K\"ahler manifolds}\label{ss:mapstoKahler} Let $\Sigma$ be a Riemannian surface and $N$ a compact K\"ahler manifold of complex dimension $n.$ The complexified tangent bundles decompose according to type:
\begin{equation}\label{complexsplitting}
\begin{split}
T\Sigma^\C = T^{1,0}\Sigma \oplus T^{0,1} \Sigma, \qquad TN^\C  = T^{1,0} N \oplus T^{0,1} N.
\end{split}
\end{equation}
Given a local holomorphic coordinate $z$ on $\Sigma,$ we have the local coordinate frames
\begin{equation*}
\frac{\p}{\p z} = \frac12 \left( \frac{\p}{\p x} - i \frac{\p}{\p y} \right), \qquad \frac{\p}{\p \bar{z}} = \frac12 \left( \frac{\p}{\p x} + i \frac{\p}{\p y} \right)
\end{equation*}
for $T^{1,0} \Sigma$ and $T^{0,1} \Sigma,$ respectively, as well as the dual frames
\begin{equation*}
dz = dx + i dy, \qquad d\bar{z} = dx - idy
\end{equation*}
for $T^*{}^{1,0} \Sigma$ and $T^*{}^{0,1} \Sigma,$ respectively.
Given local holomorphic coordinates $w^\alpha,$ $\alpha = 1, \ldots, n,$ on $N,$ we have similar formulae for the frames
$\frac{\p}{\p w^\alpha},$ $\frac{\p}{\p \bar{w}^\alpha},$ $dw^\alpha,$ and $d \bar{w}^\alpha.$

The metric tensor $g$ on $\Sigma$ may be extended complex-linearly to $TN^\C.$ We obtain a hermitian metric by the formula
\begin{equation*}
\LA v,w \RA = g(v, \bar{w}),
\end{equation*}
which agrees with the real-valued metric on $TN \subset TN^\C.$ If $z$ is a local coordinate, we may let
\begin{equation*}
\sigma(z) = 
\sqrt{2 g \left( \frac{\p}{\p z}, \frac{\p}{\p \bar{z}} \right)}.
\end{equation*}
The metric tensor and K\"ahler form on $\Sigma$ are given locally by
\begin{equation*}
g = \frac{ 1 }{2} \sigma^2 \left( dz \otimes d \bar{z} + d \bar{z} \otimes dz \right), \qquad \omega_\Sigma = \frac{i}{2} \sigma^2 dz \wedge d \bar{z}.
\end{equation*}
Similarly, we may write
\begin{equation*}
h_{\alpha \bar{\beta}} = h \left( \frac{\p}{\p w^\alpha}, \frac{\p}{\p \bar{w}^\beta} \right),
\end{equation*}
which is a Hermitian matrix. The K\"ahler form on $N$ then has the local expression
\begin{equation*}
\omega_N = i h_{\alpha \bar{\beta}} dw^\alpha \wedge d\bar{w}^\beta.
\end{equation*}
See e.g. \cite[\S 2.4]{songwaldron}, for more details.

Now, given a map $u : \Sigma \to N,$ its differential $du$ is naturally a section of $\mathcal{E} = T^*M \otimes u^* TN.$ The complexification $\mathcal{E}^\C = T^*M^\C \otimes_\C \left( u^*TN \right)^\C$ decomposes into four factors corresponding to the direct sum decompositions (\ref{complexsplitting}). The components of $du$ under this splitting can be written schematically as
\begin{equation*}
du = \left( \begin{array}{cc}
\p u & \dbar u \\[2mm]
\overline{\dbar u} & \overline{\p u} \end{array} \right),
\end{equation*}
where in local coordinates, we have
\begin{equation*}
\p u = w_z^\alpha dz \otimes \frac{\p}{\p w^\alpha}, \qquad \dbar u = w^\alpha_{\bar{z} } d\bar{z} \otimes \frac{\p}{\p w^\alpha}. 
\end{equation*}
Here we have written
$w^\alpha_z = \frac{\p w^\alpha(u(z))}{\p z},$ etc.
The Dirichlet energy density of $u,$
\begin{equation*}
e(u) = \frac12 |du|^2, 
\end{equation*}
decomposes under the above splitting as
\begin{equation*}
\begin{split}
e(u) & = \frac12 \left( |\p u|^2 + | \dbar u|^2 + |\overline{\p u}|^2 + |\overline{\dbar u} |^2 \right) \\
& = |\p u|^2 + | \dbar u|^2 \\
& =: e_\p(u) + e_\dbar(u).
\end{split}
\end{equation*}
Letting
$$E_\p(u) = \int_\Sigma e_\p(u) \, dV, \qquad E_\dbar(u) = \int_\Sigma e_{\dbar}(u) \, dV,$$
the Dirichlet energy splits as
\begin{equation*}
E(u) = E_\p(u) + E_\dbar(u).
\end{equation*}
Meanwhile, a brief calculation shows that
\begin{equation*}
u^* \omega_N = \left( e_\p(u) - e_\dbar(u) \right) \omega_{\Sigma}.
\end{equation*}
Since $d \omega_N = 0,$ the pullback is again closed and its integral is invariant under homotopy of $u.$ We conclude that
\begin{equation*}
E_\p(u) - E_\dbar(u) = \int_{\Sigma} u^* \omega_N =: 4 \pi \kappa
\end{equation*}
is a constant depending only on the homotopy class of $u.$ Rearranging, we obtain
\begin{equation}\label{chernweilidentity}
E(u) = 4 \pi \kappa + 2 E_\dbar(u).
\end{equation}
Note that if $N$ also has complex dimension one, then $\kappa =  \deg(u).$

Recall that the tension field $\T(u)$ (see (\ref{tensionfield})) is the negative $L^2$-gradient of the Dirichlet functional. We can decompose $\T(u)$ as follows:
\begin{equation*}
\T(u) = \tau(u) + \overline{\tau(u)} \in u^* T^{1,0}N \oplus u^* T^{0,1}N.
\end{equation*}
In view of (\ref{chernweilidentity}), $\tau(u)$ is half the $L^2$ gradient of the functional $E_\dbar(u),$ which can also be computed by integration-by-parts. This yields the formula
\begin{equation*}
\tau(u) = - \dbar^* \dbar u.
\end{equation*}
Here, $\dbar = \dbar_u : \Omega^{0,0}(u^* TN^{1,0} ) \to \Omega^{0,1}(u^*TN^{1,0} )$ is the coupled $\dbar$-operator using the pullback of the Levi-Civita connection on $N.$ Similarly, we have
\begin{equation*}
\tau(u) = - 2\p^* \p u.
\end{equation*}
By the K\"ahler identities, we also have
\begin{equation*}
\tau(u) = - 2\Lambda \p \dbar u = \Lambda \dbar \p u.
\end{equation*}
In local coordinates, we have the formula
\begin{equation*}
\tau(u)^\alpha = 4\sigma^{-2}(z) \left( w^\alpha_{ z \bar{z} } + {}^{N}\Gamma^{\alpha}_{\beta \gamma}  w^\beta_{z} w^\gamma_{\bar{z}} \right)
\end{equation*}
and $\T(u) = 2 \Re(\tau(u)).$ 

We have the \emph{Weitzenbock formula}, for $\alpha \in \Omega^{0,1}(u^*T^{1,0}N):$
\begin{equation}\label{weitzbeforennhbsc}
-2\LA \alpha, \dbar \dbar^* \alpha \RA = \frac12\Delta | \alpha|^2 -|\nabla \alpha|^2 - K_\Sigma |\alpha|^2 + q_1(\p u, \alpha) + q_2(\dbar u, \alpha),
\end{equation}
where
\begin{equation*}
\begin{split}
q_1(\p u, \alpha ) \stackrel{loc}{=} \sigma^{-4} {}^N R_{\beta \bar{\gamma} \delta \bar{\eta}} \left( \alpha_{\bar{z}}^\beta \overline {\alpha_{\bar{z}}^\gamma } w_z^\delta \overline{w_{z}^{\eta} } \right)
\end{split}
\end{equation*}
and
\begin{equation*}
\begin{split}
q_2(\dbar u, \alpha ) \stackrel{loc}{=} - \sigma^{-4} {}^N R_{\beta \bar{\gamma} \delta \bar{\eta}} \left( \alpha_{\bar{z}}^\beta \overline {\alpha_{\bar{z}}^\gamma } w_{\bar{z}} ^\delta \overline{w_{\bar{z}}^{\eta} } \right).
\end{split}
\end{equation*}
See \cite{schoenyauhmbook} or \cite[\S 2.6]{songwaldron} for derivations.

\vspace{2mm}

\subsection{The 2-sphere}\label{ss:sphere} The round 2-sphere $S^2 \subset \R^3$ carries a complex structure $I : TS^2 \to TS^2$ derived from the cross-product on $\R^3:$
\begin{equation*}
I_u(v) = u \times v, \quad v \perp u.
\end{equation*}
This agrees with the complex multiplication on tangent vectors coming from the identification $S^2 \cong \C \P^1,$ given by the stereographic coordinate charts
\begin{equation*}
\begin{split}
z & \mapsto \frac{1}{1 + |z|^2} \left( 2\Re z, 2\Im z, 1 - |z|^2 \right) \\
w = z^{-1} & \mapsto \frac{1}{1 + |w|^2} \left( 2\Re w, - 2\Im w, |w|^2 - 1\right).
\end{split}
\end{equation*}
A similar chart centered at any point in $S^2$ can be obtained by post-composing with an $\SO(3)$ rotation. We have
$$\sigma(z) = \frac{2}{1 + |z|^2}$$
and
$$\mathrm{Vol}_{S^2} = \sigma^2 \mathrm{Vol}_{\R^2}.$$
Letting $z = x + i y,$ and viewing $u : \C \cong \R^2 \to S^2 \subset \R^3$ as a vector-valued function on $\R^2,$ we have
\begin{equation*}
\begin{split}
e_\p(u) & = \frac{1}{4 \sigma^2} |u_x - u \times u_y |^2 \\
e_\dbar(u) & = \frac{1}{4 \sigma^2} |u_x + u \times u_y |^2 \\
\T(u) & = \frac{1}{\sigma^2} \left( \Delta_{\R^2} u + |\nabla_{\R^2} u|^2 u \right) \\
\| \T(u) \|_{L^2(S^2)} & = \| \sigma \T(u) \|_{L^2(\R^2)}.
\end{split}
\end{equation*}
See Topping \cite{toppingannals} for these formulae.



\begin{rmk} In contrast with \cite{toppingannals}, by default, we will take norms with respect to the metric and volume form of $S^2$ rather than with respect to the Euclidean metric in the stereographic chart. For $D_\rho$ with $\rho \leq C,$ these are uniformly equivalent, so the distinction can almost always be ignored. 
\end{rmk}

\vspace{10mm}

\section{Pre-\L ojasiewicz estimates}

This section contains the first several ingredients of our proof. We will assume that the domain manifold is $S^2,$ although this is not always necessary. We take the target to be a general compact Riemannian manifold, $N,$ which we will usually assume to be K\"ahler and eventually to be $S^2.$

Let $k \in \N,$ $K \geq 1,$ and $0 < \beta < \frac12.$ 
Our estimates will involve the following small constants:

\begin{itemize}

\item $\eps_0 > 0$ - constant depending only on $N$ and (later) on $k;$ or if the target is $N = S^2,$ then only on $k$

\item $\eps_1 > 0$ - constant depending on $K$ and $N$

\item $c_1 > 0$ - constant depending only on $K$

\item $\eps_2 > 0$ - constant depending only on $\beta$

\item $\eps_3 = \min_{i = 0,1,2} \eps_i.$

\end{itemize}
These will be allowed to decrease appropriately with each subsequent appearance.

\subsection{$\eps$-regularity}\label{ss:epsreg}


In this subsection, we give a standard epsilon-regularity result followed by two elementary but crucial variations needed later. Recall 
that $D_\rho(x) $ denotes the image in $S^2$ of the disk of radius $\rho$ about the origin in the stereographic chart centered at $x.$

\begin{lemma}\label{lemma:epsreg}
There exists $\eps_0 > 0,$ depending on the geometry of $N,$ as follows. Let $\sigma > 1$ and $\rho \leq 1,$ and put $D = D_\rho(x), \hat{D} = D_{\sigma \rho}(x).$

\vspace{2mm}

\noindent (a) Given $u \in W^{2,2}(\hat{D}, N)$ satisfying $\rho \| \T (u) \|_{L^2(\hat{D})} \leq 1$ and
$$E(u, \hat{D}) < \eps_0,$$
we have 
\begin{equation*}
\rho^2 \| \nabla du \|^2_{ L^2\left( D \right) } \leq C_{\sigma} \left( E(u, \hat{D} ) + \rho^2 \| \T(u) \|_{L^2(\hat{D)})}^2 \right).
\end{equation*}

\vspace{2mm}

\noindent (b) Supposing that $N$ is K\"ahler, we also have
\begin{equation*}
\rho^2 \| \nabla \p u \|^2_{ L^2 \left( D \right) } \leq C_{\sigma} \left( E_\p(u, \hat{D} ) + \rho^2 \| \T(u) \|_{L^2(\hat{D})}^2 \right)
\end{equation*}
and
\begin{equation*}
\rho^2 \| \nabla \dbar u \|^2_{ L^2\left( D \right) } \leq C_{\sigma} \left( E_\dbar (u, \hat{D} ) + \rho^2 \| \T(u) \|_{L^2(\hat{D})}^2 \right).
\end{equation*}
Similar statements hold after replacing $D \Subset \hat{D}$ by any pair of precompact open sets $\Omega \Subset \hat{\Omega}.$ 
\end{lemma}
\begin{proof} The proof of ($a$) is a minor adaptation of the proof by Ding-Tian \cite[Lemma 2.1]{dingtian}, which we omit.

To prove (b), we apply (\ref{weitzbeforennhbsc}) with $\alpha = \dbar u,$ to get
\begin{equation*}
\LA \dbar u, \dbar \tau(u) \RA \leq - |\nabla \dbar u|^2 + \frac12 \Delta | \dbar u|^2 + \left( -K_\Sigma + C |d u|^2 \right) |\dbar u|^2.
\end{equation*}
The result follows by a similar argument.
\end{proof}

\begin{lemma}\label{lemma:qlemma} Suppose that $N$ is K\"ahler. Let $\frac43 \leq q \leq 2,$ $K \geq 1,$ $0 < \rho \leq 1,$ and $x_0 \in S^2.$ Let $U = U_\rho^{2\rho}(x_0)$ and $\hat{U} = U_{\rho/2}^{4\rho}(x_0) \subset S^2.$ Suppose that $\rho \| \T(u) \| \leq 1,$ $E(u, \hat{U}) < \eps_0,$ and
\begin{equation}\label{qlemma:assumption}
E_\dbar (u,\hat{U}) + \rho^2 \|\T(u)\|^2_{L^2 \left( \hat{U} \right)} \leq K E_\dbar(u, U).
\end{equation}
Let $V \subset U$ be any measurable set with
\begin{equation}\label{qlemma:Vassumption}
|V| \geq \left( 1 - (2CK)^{-2 } \right) |U|,
\end{equation}
where $C$ is a sufficiently large universal constant. We then have
\begin{equation*}
E_\dbar(u, U) \leq C K \rho^{2 - \frac{4}{q} } \| e_{\dbar} \|_{L^{\frac{q}{2}}(V)}.
\end{equation*}
\end{lemma}
\begin{proof} By scaling-invariance, we may assume without loss of generality that $\rho = 1.$ Let
$$p = \frac{q^*}{2} = \frac{q}{2(q - 1)} \in \LB 1, 2 \RB.$$
By Lemma \ref{lemma:epsreg}$b$ and the Sobolev inequality, we have
\begin{equation}
\label{qlemma:L4bound} \|e_{\dbar} \|_{L^2(U)} = \| \dbar u \|^2_{L^4(U)} \leq C \left( E_\dbar(u, \hat{U}) + \|\T(u)\|^2_{L^2 \left( \hat{U} \right) } \right) \leq CK E_\dbar (u, U),
\end{equation}
where we have applied the assumption (\ref{qlemma:assumption}).
Then (\ref{qlemma:L4bound}) and H\"older's inequality give
\begin{equation}\label{qlemma:edbar}
\| e_{\dbar} \|_{L^{p}(U)} \leq C  \| e_\dbar \|_{L^2(U)} \leq C K E_\dbar (u, U).
\end{equation}
We have
\begin{equation*}
\begin{split}
E_\dbar(u, U) & = E_\dbar(u, U \setminus V) + E_\dbar(u, V).
\end{split}
\end{equation*}
Applying H{\"o}lder's inequality on the first term and the interpolation inequality on the second, we obtain
\begin{equation*}
\begin{split}
E_\dbar(u, U) & \leq  |U \setminus V|^{\frac12 } \| e_{\dbar} \|_{L^2(U) } + \| e_\dbar \|_{L^p(V) }^{\frac12} \| e_{\dbar} \|_{L^{\frac{q}{2} }(V) }^{\frac12} .
\end{split}
\end{equation*}
By (\ref{qlemma:Vassumption}), we have
\begin{equation}\label{qlemma:Vcarea}
|U \setminus V|^{\frac12 } \leq \frac{1}{2CK} .
\end{equation}
Inserting (\ref{qlemma:edbar}-\ref{qlemma:Vcarea}), we obtain
 \begin{equation*}
\begin{split}
E_\dbar(u, U) 
&  \leq \frac12 E_\dbar(u, U) + \left( C K  E_\dbar(u, U) \right)^{\frac12} \| e_{\dbar} \|_{L^{ \frac{q}{2} }(V) }^{\frac12}.
\end{split}
\end{equation*}
Rearranging and cancelling exponents, we have the desired estimate.
\end{proof}



\begin{lemma}\label{lemma:supinf} Let $N, U,$ and $\hat{U}$ be as in the previous lemma. Given $K \geq 1,$ there exist $\eps_1 > 0,$ depending on $K$ and $N,$ and $c_1 > 0,$ depending on $K,$ as follows. Let $u: \hat{U} \to N$ be a map that satisfies
$$ E(u, \hat{U} ) < \eps_1,$$
\begin{equation}\label{supinf:Tassumption}
\rho^2 \| \T (u) \|^2_{L^2(\hat{U} )} < \eps_1 E_\p (u, U),
\end{equation}
and
\begin{equation}\label{supinf:assumption}
E_\p(u,\hat{U} ) \leq K E_\p (u, U).
\end{equation}
Then
\begin{equation}\label{supinf:supinf}
c_{1} \sup_{\rho \leq r \leq 2\rho} \int_{S^1_r} e_\p \, d\theta \leq \inf_{\rho \leq r \leq 2\rho} \int_{S^1_r} e_\p \, d\theta
\end{equation}
and
\begin{equation}\label{supinf:area}
\mathrm{Area} \left\{ x \in U \mid \rho^{2} e_\p(u)(x) \leq c_1 E_\p(u, U) \right\} \leq (2CK)^{-2} \mathrm{Area}(U).
\end{equation}
\end{lemma}
\begin{proof} This will follow by applying a standard contradiction argument to the $(1,0)$-form $\p u.$ We may assume $\rho = 1$ without loss of generality. We view $U$ and $\hat{U}$ as subsets of the ball $B_{4}(0) \subset \C$ by stereographic projection.

First let $\alpha = f(z) dz$ be a $\C$-valued holomorphic 1-form on $U' = U^{3}_{2 / 3},$ satisfying $$\dbar \alpha = 0, \quad \int_U |\alpha|^2 = 1,
\,\,\,\, \text{and} \,\,\,\,
\int_{U'} |\alpha|^2 \leq K.$$
By elementary complex analysis, $\alpha$ must obey estimates of the form
\begin{equation}\label{supinf:supinfcomplexanalysisbound}
2 c_1 \sup_{1 \leq r \leq 2} \int_{S^1_r} |\alpha|^2 \, d\theta \leq \inf_{1 \leq r \leq 2} \int_{S^1_r} |\alpha|^2 \, d\theta
\end{equation}
and
\begin{equation*}
\mathrm{Area} \left\{ x \in U \mid |\alpha(x)|^2 \leq 2 c_{1}  \right\} \leq (2CK)^{-2}
\end{equation*}
for a small constant $c_1 > 0$ depending on $K.$ The same statements hold for a holomorphic 1-form valued in a flat holomorphic bundle over $U'.$

We can now prove (\ref{supinf:supinf}) by contradiction. Assume that $u_i : \hat{U} \to N$ is a sequence of nonconstant maps such that
\begin{equation*}
E(u_i, \hat{U} )  \leq \frac{1}{i},
\end{equation*}
$$E_\p(u_i,\hat{U} ) \leq K E_\p (u_i, U),$$
and
\begin{equation}\label{supinf:Tutozero}
\| \T (u_i) \|^2_{L^2(\hat{U} )} \leq \frac{E_\p (u_i, U)}{i},
\end{equation}
but for which
\begin{equation}\label{supinf:contradiction}
c_{1} \sup_{\rho \leq r \leq 2\rho} \int_{S^1_r} |\partial u_i|^2 \, d\theta > \inf_{\rho \leq r \leq 2\rho} \int_{S^1_r} |\partial u_i|^2 \, d\theta.
\end{equation}
By Lemma \ref{lemma:epsreg}$a,$ we have
\begin{equation*}
\| \nabla d u_i \|_{L^2(U')}^2\leq C \left( E(u_i, \hat{U}) + \frac{1}{i}E_{\p}(u_i, \hat{U}) \right) \leq C E(u_i, \hat{U}) \to 0.
\end{equation*}
Passing to a subsequence, we may assume that the images $u_i \left( U' \right)$ are contained in a fixed holomorphic coordinate neighborhood inside $N.$ We have $d u_i \to 0$ strongly in $W^{1,2}(U').$

Now, by Lemma \ref{lemma:epsreg}$b$ and (\ref{supinf:Tutozero}), we have
\begin{equation}\label{supinf:Lpboundondu}
\| \nabla \p u_i \|^2_{L^2(U')} \leq C \left( 1 + \frac{1}{i} \right) E_{\p}(u_i, \hat{U}) \leq C E_{\p}(u_i, U).
\end{equation}
Let
$$\alpha_i = \frac{\p u_i}{\sqrt{E_\p (u, U)} }.$$
We then have
$$\int_U |\alpha_i|^2 = 1,$$
$$\int_{\hat{U} } |\alpha_i|^2 \leq K,$$
$$\int_{\hat{U}} | \dbar_{u_i} \alpha_i |^2 = \frac{\int |\T(u_i)|^2}{E_\p(u, U)} \to 0,$$
and, by (\ref{supinf:Lpboundondu}),
\begin{equation*}
\| \nabla \alpha_i \|_{L^2(U')} \leq C.
\end{equation*}
We may pass to a weak limit in $W^{1,2}(U'),$
$$\alpha_i \rightharpoonup \alpha,$$
which again satisfies
\begin{equation}\label{supinf:boundsonalpha}
\begin{split}
\int_U |\alpha|^2 & = 1, \\
\int_{U' } |\alpha|^2 & \leq K.
\end{split}
\end{equation}
In the holomorphic coordinates on $N,$ we have
$$\left( \dbar_{u_i} \alpha_i \right)^\mu = \dbar \left( \alpha_i \right)^\mu + {}^N \Gamma(du_i, du_i) \# \alpha_i.$$
Since the left-hand side tends to zero in $L^2,$ 
while $du_i$ tends to zero 
and $\alpha_i$ is bounded in $L^p$ for all $p,$
$ \dbar (\alpha_i)^\mu$
must also tend to zero in $L^2;$ hence $\dbar \alpha = 0.$ Moreover, since $\alpha_i$ converges weakly in $W^{1,2}(U'),$ $\sqrt{\int_{S^1_r} |\alpha_i|^2 \, d\theta}$ converges weakly in $W^{1,2}(\LB \rho, 2\rho \RB)$ as a function of $r,$ and in particular strongly in $C^0(\LB \rho, 2\rho \RB),$ so (\ref{supinf:contradiction}) is preserved.
But, together with (\ref{supinf:boundsonalpha}), this contradicts (\ref{supinf:supinfcomplexanalysisbound}), establishing (\ref{supinf:supinf}).

The proof of (\ref{supinf:area}) is similar.
\end{proof}


\vspace{5mm}

\subsection{Topping's estimates}


The following basic estimates are due to Topping \cite{toppingannals}. Throughout this subsection and the next, we suppose that $u$ maps from $S^2$ to $S^2.$

\begin{prop}[\cite{toppingannals}, Lemma 2.5b]\label{prop:globalrepulsion} For $1 \leq q < 2,$ assuming $E(u) \leq 4 \pi k,$ we have
$$ \| \sqrt{e_\p(u) e_\dbar(u)} \|_{L^q(S^2)} \leq C_q \sqrt{k} \| \T (u) \|_{L^2(S^2)} .$$
\end{prop}

\begin{prop}[\cite{toppingannals}, Lemma 2.3d]\label{prop:localdominanceoftension} Let $0 < \rho \leq R < \infty,$ with $R \geq 1,$ and assume that $D_{\rho_i}(x_i) \subset D_{\rho}(x_0)$ for each $i.$ Let
$$\Omega' = D_{\rho}(x_0) \setminus \cup_i \bar{D}_{\rho_i}(x_i), \quad \hat{\Omega}' = D_{2R}(x_0) \setminus \cup_i \bar{D}_{\rho_i/2}(x_i).$$  
If $E_{\dbar}(u, \hat{\Omega}' ) < \eps_2 = \eps_2(\beta),$
then we have
\begin{equation*}
\begin{split}
E_\dbar(u, \Omega' ) \leq C_{\ell, \beta} \left( \frac{\rho}{R} \right)^\beta \left( R^2 \| \T(u) \|_{L^2 \left( \hat{\Omega}' \right) }^2 + E_\dbar \left( u, U^{2R}_{R}(x_0) \right)  + \sum_{i = 1}^\ell \left(\frac{R}{\rho_i} \right)^{\beta} E_\dbar ( u , U^{\rho_i }_{\rho_i/2} (x_i) ) \right).
\end{split}
\end{equation*}
\end{prop}
\begin{proof} After homothetically rescaling so that $R = 1,$ the $E_\dbar$ terms are unchanged and the $R^2 \| \T(u) \|^2$ term decreases; hence, we may assume without loss of generality that $R = 1.$ Letting $\mu = 1$ in Topping's notation, the proof of \cite[Lemma 2.3d]{toppingannals} on p. 486 gives
\begin{equation*}
\begin{split}
E_\dbar(u, \Omega' ) & \leq C \rho^{\frac{2 (q - 1)}{q}} \| \dbar u \|_{L^{\frac{2q}{2 - q}}(\hat{\Omega}') } \\
& \leq C_{q} \rho^{\frac{2 (q - 1)}{q}} \left( \| \T(u) \|_{L^2 \left( \hat{\Omega}' \right) } + \| \dbar u \|_{L^{q}(U^{2}_{1}) }  + \sum_{i = 1}^\ell \frac{1}{\rho_i} \| \dbar u \|_{L^q(U^{\rho_i }_{\rho_i/2} (x_i) ) } \right).
\end{split}
\end{equation*}
The proof proceeds by applying H\"older's inequality on the remaining terms, as in \cite{toppingannals}, and letting $\beta = \frac{4(q - 1)}{q}.$
\end{proof}

\begin{prop}\label{prop:globaldominanceoftension} Assume $0 < \rho_i \leq \frac{1}{4\sqrt{\ell}}$ for $i = 1, \ldots, \ell.$ Write 
\begin{equation*}
\Omega = S^2 \setminus \cup_i D_{\rho_i }(x_i), \quad 
\hat{\Omega} = S^2 \setminus \cup_i D_{\rho_i / 2}(x_i).
\end{equation*}
Supposing that $E_\dbar(u, \hat{\Omega} ) < \eps_0,$ we have
$$E_\dbar(u, \Omega ) \leq C_{\ell, \beta} \left( \| \T(u) \|_{L^2(\hat{\Omega})}^2 + \sum \rho_i^{-\beta} E_\dbar ( u , U^{\rho_i}_{\rho_i/2} (x_i) ) \right).$$
\end{prop}
\begin{proof} 
Since $\rho_i \leq \frac{1}{4\sqrt{\ell}},$ there exist two antipodal points $x_0, \hat{x}_0 \in S^2$ such that
$$D_{1/ 4\sqrt{\ell} }(x_0) \cap D_{\rho_i}(x_i) = \emptyset = D_{1/ 4\sqrt{\ell} }(\hat{x}_0) \cap D_{\rho_i}(x_i)$$
for $i = 1, \ldots, \ell.$ Let
$$\rho = 4\sqrt{\ell}, \qquad R = (2 C_{\ell,\beta} )^{\frac{1}{\beta}} \rho,$$
where $C_{\ell,\beta} \geq 1$ is the constant of the previous Proposition.
We apply the proposition twice, 
to obtain
\begin{equation*}
\begin{split}
E_\dbar(u, \Omega \cap D_{4\sqrt{\ell}} (x_0) ) \leq \frac{1}{2} \left( R^2 \| \T(u) \|_{L^2 (S^2) }^2 + E_\dbar \left( u, U^{2R}_{R}(x_0) \right) + \sum_{i = 1}^\ell \left(\frac{R}{\rho_i} \right)^{\beta} E_\dbar ( u , U^{\rho_i }_{\rho_i/2} (x_i) ) \right),
\end{split}
\end{equation*}
\begin{equation*}
\begin{split}
E_\dbar(u, \Omega \cap D_{4\sqrt{\ell}}(\hat{x}_0) ) \leq \frac{1}{2} \left( R^2 \| \T(u) \|_{L^2 (S^2) }^2 + E_\dbar \left( u, U^{2R}_{R}(\hat{x}_0) \right)  + \sum_{i = 1}^\ell \left(\frac{R}{\rho_i} \right)^{\beta} E_\dbar ( u , U^{\rho_i }_{\rho_i/2} (x_i) ) \right).
\end{split}
\end{equation*}
Since $R \geq 4\sqrt{\ell},$ we have $U^{2R}_R(x_0) \subset D_{\frac{1}{4\sqrt{\ell} }}(\hat{x}_0),$ so $U^{2R}_R(x_0) \subset \Omega \cap D_{4\sqrt{\ell}}(\hat{x}_0);$ also $U^{2R}_R(\hat{x}_0) \subset D_{\frac{1}{4\sqrt{\ell} }}(x_0),$ so $U^{2R}_R(\hat{x}_0) \subset \Omega \cap D_{4\sqrt{\ell}} (x_0).$ Therefore
$$E_\dbar \left( u, U^{2R}_{R}(\hat{x}_0) \right) +  E_\dbar \left( u, U^{2R}_{R}(x_0) \right) \leq E_\dbar(u, \Omega).$$
We may add the above two inequalities together and rearrange, to obtain
\begin{equation*}
\begin{split}
& E_\dbar(u, \Omega ) \leq 2 \left( R^2 \| \T(u) \|_{L^2 (S^2) }^2 + \sum_{i = 1}^\ell \left( \frac{R}{\rho_i} \right)^{\beta} E_\dbar ( u , U^{\rho_i }_{\rho_i/2} (x_i) ) \right) .
\end{split}
\end{equation*}
Finally, we absorb $R$ into the constant to obtain the desired estimate.
\end{proof}

\begin{lemma}\label{lemma:comparisonmap}

Let $\Omega$ be as above and $u : \Omega \to S^2.$ Let $U_i = U^{2 \rho_i}_{\rho_i}(x_i)$ for $i = 1, \ldots, \ell.$ Suppose that
\begin{equation*}
\sum_{i = 1}^\ell E(u, U_i) + \|\T(u)\|_{L^2(\Omega)}^2 \leq \eps < \eps_0.
\end{equation*}
There exists a map $w : S^2 \to S^2$ which is constant outside $\Omega,$ agrees with $u$ on $S^2 \setminus \cup_i D_{2\rho_i}(x_i),$ and satisfies
\begin{equation*}
\sum_{i = 1}^\ell E(w, U_i) + \| \T(w) \|^2 \leq C \eps.
\end{equation*}
\end{lemma}
\begin{proof}
It follows from Lemma \ref{lemma:epsreg}$a$ that the oscillation of $u$ over $U^{5\rho_i/3}_{4\rho_i/3} (x_i) \Subset U_i$ is small. We may therefore apply the argument of \cite[Lemma 2.22]{toppingannals} to construct a map $w$ with the required bound on the energy. The bound on the tension of $w$ also follows from the $W^{2,2}$ bound of Lemma \ref{lemma:epsreg}$a.$
\end{proof}

\begin{lemma}\label{lemma:cutandpaste}

Let $u$ be as in the previous Lemma and suppose that
\begin{equation*}
\sum_{i = 1}^\ell E(u, U_i) + E_\dbar(u, \Omega) + \|\T(u)\|^2 \leq \eps < \eps_0.
\end{equation*}
Then
\begin{equation*}
\dist \left( E(u, \Omega), 4\pi\Z  \right) \leq C \eps.
\end{equation*}

\end{lemma}
\begin{proof}
This follows from the previous Lemma and \cite[Lemma 2.23]{toppingannals}.
\end{proof}

\vspace{5mm}

\subsection{Pre-\L ojasiewicz estimates}

This subsection contains the rudimentary versions of our main estimates. We continue to assume that both the domain and the target are $S^2.$

\begin{thm}[Pre-\L ojasiewicz estimate]\label{thm:preloj} 
Given $k \in \N,$ $K \geq 1,$ and $0 < \beta < \frac12,$ there exists $\eps_3 > 0$ as follows.

Let $x_i \in S^2$ and $0 < \rho_i \leq \lambda  \leq \frac{\pi}{2},$ for $i = 1, \ldots, \ell,$ with $\ell \leq k,$ be such that $D_{2\rho_i}(x_i) \cap D_{2\rho_j}(x_j) = \emptyset$ for $i \neq j.$ Write
$$U_i = U^{2 \rho_i}_{\rho_i}(x_i), \qquad \hat{U}_i = U^{4 \rho_i}_{\rho_i/2}(x_i),$$
\begin{equation*}
\Omega = S^2 \setminus \cup_i D_{\rho_i }(x_i), \qquad
\hat{\Omega} = S^2 \setminus \cup_i D_{\rho_i / 2}(x_i).
\end{equation*}

\vspace{2mm}

\noindent ($a$) Suppose that $u \in W^{2,2} (S^2, S^2)$ satisfies
$$E(u) \leq 4\pi k,$$
$$\|\T(u) \| \leq 1,$$
\begin{equation}\label{preloj:Edassumption} 
\max \{ E(u, \hat{U}_i) , E_\dbar (u, \hat{\Omega} ) \} < \eps_3,
\end{equation}
and, for some $M \geq 1$ and all $i = 1, \ldots, \ell,$
\begin{equation}
\label{preloj:edassumption} E(u, \hat{U}_i) \leq K  E(u, U_i) + M \rho_i^2 \| \T(u) \|_{L^2(\hat{U}_i)}^2 ,
\end{equation}
and
\begin{equation}\label{preloj:MEdEdbarassumption}
M^{-1} \left(  E_{\p}(u_i, U_i ) - \rho_i^2 \| \T(u) \|_{L^2(U_i)}^2 \right) \leq E_{\dbar}(u, U_i) \leq M \left(  E_\p(u, U_i) + \rho_i^2 \| \T(u) \|_{L^2(\hat{U}_i)}^2 \right).
\end{equation}
 Then
\begin{equation*}
\sum_{i=1}^\ell \rho_i^{-\beta} E\left( u, U_i \right) + \dist \left( E(u, \Omega) , 4 \pi \Z \right) \leq C_{  \ref{thm:preloj} a } \left( \|\T(u) \|^2 + \lambda^{1 - 2 \beta } \| \T(u) \| \right).
\end{equation*}

\vspace{2mm}

\noindent ($b$) Let $m \geq 2$ and
$$1 < \alpha = \frac{m + \beta}{m - 1 + 3\beta} <2.$$
Let $j \in \{1, \ldots, \ell\}$ be such that $\rho_{j} = \max_i \rho_i.$ Suppose further that
\begin{equation}\label{preloj:bassumption}
\rho_{j}^{m + 2\beta} \leq M \left( E_\p(u, U_j) + \rho_j^2 \| \T(u) \|^2 \right).
\end{equation}
We then have
\begin{equation}\label{preloj:rhoiestimate}
\max_i \rho_i^{m - 1 + 3\beta} \leq C_{\ref{thm:preloj} b} \| \T(u) \|
\end{equation}
and
\begin{equation}\label{preloj:lojestimate}
\sum_{i=1}^\ell \rho_i^{-\beta} E\left( u, U_i \right) + \dist \left( E(u, \Omega) , 4 \pi \Z \right) \leq C_{  \ref{thm:preloj} b } \|\T(u) \|^\alpha.
\end{equation}
Here $C_{ \ref{thm:preloj}a-b}$ depend on $k, K, M,$ and $\beta.$
\end{thm}
\begin{proof} 
Let $\eps_3$ be the minimum of the constants $\eps_i$ from Lemma \ref{lemma:epsreg} and Propositions \ref{prop:localdominanceoftension}-\ref{prop:globaldominanceoftension}, and let $c_1$ be the constant of Lemma \ref{lemma:supinf} corresponding to the given $K.$
Let $q = \frac{2}{1 + \beta},$ so that $\frac43 \leq q < 2$ and
$$-\beta = 1 - \frac{2}{q}, \qquad 1 - 2 \beta = 3 - \frac{4}{q}, \qquad 2 - \beta =  3 - \frac{2}{q}.$$

\vspace{2mm}

\noindent ($a$) We will first prove
\begin{equation}\label{preloj:claim}
\lambda^{3 - \frac{4}{q} } \| \T \| + \lambda^{3 - \frac{2}{q} } \| \T \|^2 \gtrsim \sum_i \rho_i^{1 - \frac{2}{q}} E_\dbar(u, U_i).
\end{equation}
Let 
$$J = \left\{ i \in \{1, \ldots, \ell\} \mid \rho_i^2 \|\T \|^2_{L^2 \left( \hat{U}_i \right)} < \frac{\eps_3}{M^2} E_\dbar(u, U_i) \right\}.$$
Summing over $i \in J^c,$ we have
\begin{equation}\label{preloj:Jest}
\sum_{i \in J^c} \rho_i^{1 - \frac{2}{q} } E_\dbar(u, U_i) \leq \lambda^{3 - \frac{2}{q} } \sum_{i \in J^c} \rho_i^{-2} E_\dbar(u, U_i) \leq \frac{M^2 \lambda^{3 - \frac{2}{q} } }{\eps_3}  \|\T \|^2.
\end{equation}
On the other hand, for $i \in J,$ we have
\begin{equation*}
\rho_i^2 \|\T \|^2_{L^2 \left( \hat{U}_i \right)} < \frac{\eps_3}{M^2} E_\dbar(u, U_i).
\end{equation*}
Combined with (\ref{preloj:edassumption}-\ref{preloj:MEdEdbarassumption}), this gives the 
conditions (\ref{qlemma:assumption}) and (\ref{supinf:Tassumption}-\ref{supinf:assumption}).

Let
$$V_i = \left\{ x \in U_i \mid \rho_i^2 e_\p(u)(x) \geq c_{1} E_\p(u, U_i)  \right\}.$$
By Lemma \ref{lemma:qlemma}, we have
\begin{equation*}
E_\dbar(u, U_i) \leq C K \rho_i^{2 - \frac{4}{q} } \| e_{\dbar} \|_{L^{\frac{q}{2}}(V_i)}.
\end{equation*}
By (\ref{supinf:area}), we have
\begin{equation*}
|V_i| \geq \left( 1 - (CK)^{-2 } \right) |U_i|.
\end{equation*}
Applying Proposition \ref{prop:globalrepulsion}, we obtain
\begin{equation*}
\begin{split}
C k\| \T \|^2 & \geq \| e_\p e_\dbar \|_{L^{\frac{q}{2}}(S^2) } \\
&  \geq \frac{1}{\ell^{\frac{q}{2}}} \sum_i \| e_\p e_\dbar \|_{L^{\frac{q}{2}}(V_i) } \\
 & \geq \frac{c_{1} }{\ell^{\frac{q}{2}} } \sum_i \rho_i^{-2} E_{\p}(u, U_i) \| e_\dbar \|_{L^{\frac{q}{2}}(V_i) }  \\
& \geq \frac{c_{1} }{MK \ell^{\frac{q}{2}}} \sum_i \rho_i^{ \frac{4}{q} - 4 } E_{\dbar}(u, U_i)^2 \\
& \geq \frac{c_{1} }{MK \ell^{\frac{q}{2} + 1} } \left( \sum_i \rho_i^{\frac{2}{q} - 2 } E_{\dbar}(u, U_i) \right)^2.
\end{split}
\end{equation*}
Rearranging and canceling squares, we get
\begin{equation}\label{preloj:Jcest}
\begin{split}
\sqrt{\frac{C MKk \ell^{\frac{q}{2} + 1} }{c_{1} } } \lambda^{3 - \frac{4}{q} } \| \T \| & \geq \sum_i \rho_i^{1 - \frac{2}{q} } E_{\dbar}(u, U_i) .
\end{split}
\end{equation}
Combining (\ref{preloj:Jest}) and (\ref{preloj:Jcest}) now yields the claim (\ref{preloj:claim}).

Applying Proposition \ref{prop:globaldominanceoftension}, we obtain
\begin{equation*}
 E_\dbar(u, \Omega ) \leq C \left( \|\T \|^2 + \lambda^{1 - 2 \beta } \| \T \| \right).
\end{equation*}
The desired estimate now follows from Lemma \ref{lemma:cutandpaste}. 

\vspace{2mm}

\noindent ($b$) We may take $\lambda = \rho_{j}.$ By our assumption (\ref{preloj:bassumption}) and ($a$), we have
\begin{equation*}
\begin{split}
\lambda^{m + \beta} \leq M \left( \lambda^{-\beta} E(u, U_j) + \lambda^{2 - \beta} \| \T \|^2_{L^2(\hat{U}_j)} \right) & \leq C \left( \| \T \| +  \lambda^{1 - 2\beta} \right) \| \T \|
\end{split}
\end{equation*}
and
\begin{equation*}
\begin{split}
0 & \leq  \| \T \|^2 +  \lambda^{1 - 2\beta} \| \T \| - \frac{\lambda^{m + \beta}}{C}.
\end{split}
\end{equation*}
Adding $\left( \lambda^{1 - 2\beta} - \frac{\lambda^{m - 1 + 3\beta}}{2C} \right) \| \T \|  \geq 0$ to the RHS and factoring, we have
\begin{equation*}
\begin{split}
0 & \leq  \left( \| \T \| - \frac{ \lambda^{m - 1 + 3\beta} }{2C} \right) \left( \| \T \| + 2 \lambda^{1 - 2 \beta} \right),
\end{split}
\end{equation*}
which gives the first estimate of ($b$). We therefore have
\begin{equation*}
\begin{split}
\lambda^{1 - 2 \beta} \| \T \| & \leq C  \| \T \|^{\frac{1 - 2 \beta}{m + 3 \beta - 1} + 1} \\
& = C \| \T \|^{\frac{m + \beta}{m + 3 \beta - 1}}.
\end{split}
\end{equation*}
The second estimate of ($b$) now follows from ($a$).
\end{proof}

\begin{thm}[Pre-\L ojasiewicz estimate, local version]\label{thm:localloj}
Let $x_i \in S^2$ and $0 < \rho_i \leq \lambda \leq \frac{\pi}{2},$ for $i = 0, \ldots, \ell,$ be such that
$$D_{2\rho_i}(x_i) \cap D_{2\rho_j}(x_j) = \emptyset$$
for nonzero $ i \neq j,$ and
$$\bigcup_{i = 1}^\ell D_{2 \rho_i}(x_i) \subset D_{\rho_0} (x_0).$$
Write $U_i = U^{2 \rho_i}_{\rho_i}(x_i),$ $\hat{U}_i = U^{4 \rho_i}_{\rho_i/2}(x_i),$ and
\begin{equation*}
\begin{split}
\Omega' & = D_{\rho_0 }(x_0) \setminus \cup D_{\rho_i }(x_i) \\
\hat{\Omega}' & = D_{2\rho_0 }(x_0) \setminus \cup D_{\rho_i / 2}(x_i).
\end{split}
\end{equation*}

\vspace{2mm}

\noindent (a) Let $u \in W^{2,2}(S^2, S^2)$ with $\| \T(u) \|_{L^2(\hat{\Omega}' )} \leq \rho_0^{-1},$ and make the assumptions (\ref{preloj:Edassumption}-\ref{preloj:MEdEdbarassumption}) for $i = 0, \ldots, \ell.$ We then have
\begin{equation*}
\sum_{i=0}^\ell \rho_i^{-\beta} E\left( u, U_i \right) + \dist \left( E\left( u, \Omega' \right), \Z \right) \leq C_{\ref{thm:localloj}a}\left( \lambda^{2 - \beta} \|\T(u) \|^2+ \lambda^{1 - 2 \beta } \| \T(u) \| \right).
\end{equation*}
Here $C_{\ref{thm:localloj}a}$ depends on $k, K, M,$ and $\beta.$

\vspace{2mm}

\noindent (b) Let $\alpha$ be as in Theorem \ref{thm:preloj}$b.$ Suppose further that
$$\rho_0^{m + 2\beta} \leq M \left( E(u, U_0) + \rho_0^2 \| \T (u) \|_{L^2(\hat{\Omega}' )} ^2 \right).$$
Then 
\begin{equation*}
\sum_{i=0}^\ell \rho_i^{-\beta} E\left( u, U_i \right) + \dist \left( E\left( u, \Omega' \right), \Z \right) \leq C_{\ref{thm:localloj}b} \| \T(u) \|^\alpha.
\end{equation*}
Here $C_{\ref{thm:localloj}b}$ depends on $k, K, M,\alpha,$ and $\beta.$
\end{thm}
\begin{proof}
Applying Proposition \ref{prop:localdominanceoftension} in place of Proposition \ref{prop:globaldominanceoftension} in the previous proof, together with (\ref{preloj:claim}), we get 
\begin{equation*}
E_\dbar(u, \Omega') \leq C\left( \lambda^2 \| \T \|^2 + \lambda^{2 - \beta} \|\T \|^2 + \lambda^{1 - 2 \beta } \| \T \| \right).
\end{equation*}
The desired estimate (a) follows from Lemma \ref{lemma:cutandpaste}. 
The estimate (b) follows as in Theorem \ref{thm:preloj}$b.$
\end{proof}

\vspace{2mm}



\vspace{10mm}

\section{Estimates on energy density in neck regions}

In this section, we prove the lower and upper bounds on the the (anti)holomorphic energy densities that we will later assemble into the needed estimates on an almost-holomorphic bubble tree.

\subsection{Three-annulus estimate}

\begin{lemma}\label{lemma:3int} Suppose that $N$ is K\"ahler. Let $n \in \Z$ with $|n| \leq L,$ $1 < \tau \leq \sigma,$ and $0 < \beta \leq \frac12$ be such that
\begin{equation}\label{3int:betassn}
\frac{2 \sigma^2}{\sigma^2 + 1}  < \sigma^{2 \beta}.
\end{equation}
There exists $\xi_0 > 0,$ depending on $N, L, \tau, \sigma,$ and $\beta,$ as follows.

Given any $x_0 \in S^2$ and $0 < \rho \leq 1,$ let $\hat{U} = U^{\tau \sigma \rho}_{\sigma^{-1} \rho }(x_0).$ Suppose that $u: \hat{U}  \to N$ is a nonconstant $W^{2,2}$ map with
\begin{equation}\label{3int:Eassumption}
E(u, \hat{U}) < \xi_0
\end{equation}
and
\begin{equation}\label{3int:tensionassumption}
\rho^{2}\| \T(u) \|^2_{L^2(\hat{U})} < \xi_0 E_\p(u, \hat{U}) .
\end{equation}
Let
$$S(1) = \sup_{\sigma^{-1} \! \rho < r \leq \tau \sigma^{-1} \! \! \rho } r^{-n} \sqrt{ \dashint_{S^1_r(x_0)} e_\p \, d\theta }.$$
$$S(2) = \sup_{\rho \leq r \leq \tau \! \rho } r^{-n} \sqrt{ \dashint_{S^1_r(x_0)} e_\p \, d\theta }.$$
$$S(3) = \sup_{\sigma \! \rho \leq r < \tau \sigma \! \rho } r^{-n} \sqrt{ \dashint_{S^1_r(x_0)} e_\p \, d\theta }.$$
The following implications hold:
\begin{equation}\tag{$a$}\label{3int:1stimpl} S(1) \leq \sigma^{1 - \beta}  S(2) \Rightarrow S(2) < \sigma^{\beta}  S(3)
\end{equation}
\begin{equation}\tag{$b$}\label{3int:2ndimpl}  S(3) \leq\sigma^{1 - \beta}  S(2) \Rightarrow S(2) <  \sigma^\beta S(1).
\end{equation}
\end{lemma}
\begin{proof} We first check the implications when $\p u = g(z) dz $ is a nonzero holomorphic 1-form on a flat annulus $V = U^{\tau \sigma}_{\sigma^{-1}}(0) \subset \C.$  More specifically, we prove:
\begin{equation}\tag{$a'$} S(1) \leq \sqrt{\sigma \left( 1 + \frac{(\sigma - 1)^2}{2\sigma} \right) } S(2) \Rightarrow S(2) \leq \sqrt{ \frac{\sigma }{1 + \frac{(\sigma - 1)^2}{2\sigma} } } S(3)
\end{equation}
\begin{equation}\tag{$b'$}  S(3) \leq \sqrt{\sigma \left( 1 + \frac{(\sigma - 1)^2}{2\sigma} \right) } S(2) \Rightarrow S(2) \leq \sqrt{ \frac{\sigma }{1 + \frac{(\sigma - 1)^2}{2\sigma} } } S(1)
\end{equation}
Under the assumption (\ref{3int:betassn}), these implications are strictly stronger than ($a$) and ($b$).

Dividing $g$ by $S(2) z^n,$ we may reduce to the case $n = 0$ and $S(2) = 1.$
To prove ($a'$) in this case, we use the Laurent expansion of $g,$ which reads
$$g(z) = \sum_{n = -\infty}^\infty a_n z^n.$$
This gives
\begin{equation}\label{defnofF}
F(r)^2 : = \frac{1}{2 \pi }  \int_{S^1_r} |g|^2 \, d\theta = \sum |a_n|^2 r^{2n}.
\end{equation}
Assume that the supremum $S(2) =1$ is attained at $r_0 \in \LB 1, \tau \RB,$ so
\begin{equation}\label{Fsupremum}
F(r_0)^2 = 1 = \sum |a_n|^2 r_0^{2n}.
\end{equation}
We have the identity
\begin{equation*}
1 = \frac{\sigma^{-1} + \sigma}{2 + \frac{(\sigma - 1)^2}{\sigma} },
\end{equation*}
and, for $n \in \Z,$ the inequality
\begin{equation}\label{sigmaidentity}
1 \leq \frac{\sigma^{-2n - 1} + \sigma^{2n + 1} }{2 + \frac{(\sigma - 1)^2}{\sigma}}  .
\end{equation}
Inserting (\ref{sigmaidentity}) into (\ref{Fsupremum}) for each $n \in \Z,$ we have
\begin{equation*}
\begin{split}
1 = \sum |a_n|^2 r_0^{2n} & \leq \frac{1}{2 + \frac{(\sigma - 1)^2}{\sigma}}  \sum |a_n|^2 r_0^{2n} \left( \sigma^{-2n - 1}  + \sigma^{2n + 1}\right) \\
& = \frac{1}{2 + \frac{(\sigma - 1)^2}{\sigma}} \left( \sigma^{-1} F \left( \sigma^{-1} r_0 \right)^2 + \sigma F \left( \sigma r_0  \right)^2 \right) \\
& \leq \frac{1}{2 + \frac{(\sigma - 1)^2}{\sigma}} \left( \sigma^{-1} S(1)^2 + \sigma S(3)^2 \right).
\end{split}
\end{equation*}
We insert the assumption $S(1)^2 \leq \sigma \left( 1 + \frac{(\sigma - 1)^2}{2\sigma} \right) ,$ and rearrange, to obtain
\begin{equation*}
1\leq \frac{2 \sigma }{2 + \frac{(\sigma - 1)^2}{\sigma}} S(3)^2,
\end{equation*}
establishing ($a'$). The proof of ($b'$) is identical.

The proof of ($a$) and ($b$) now follows from a standard compactness-and-contradiction argument, similar to the proof of Lemma \ref{lemma:supinf}. Suppose that that the implication ($a$) fails for any choice of constant $\xi_0.$ Then we can choose a sequence of maps $u_i : \hat{U} \to N$ with $E(u_i, \hat{U}) \leq \frac{1}{i}$ and
\begin{equation}\label{3int:smalltension}
\| \T(u_i) \|_{L^2(\hat{U})}^2 \leq \frac{E_{\p}(u_i, \hat{U})}{i},
\end{equation}
but for which
\begin{equation}\label{3int:contrad}
\max \{ \sigma^{\beta - 1} S(1), \sigma^\beta S(3) \} \leq S(2)
\end{equation}
for each $i,$ where $S(j)$ is defined with respect to $u_i.$ Let $U'$ be any domain with $$U := U_{\tau \sigma^{-1}}^\sigma \Subset U' \Subset U_{\sigma^{-1}}^{\tau \sigma} = \hat{U}.$$
We have
$$E_\p(u_i, \hat{U})  \leq C_\sigma \left( S(1)^2 + S(3)^2 + E_\p(u_i, U ) \right) \leq C_\sigma \left( S(2)^2 + E_\p(u_i, U ) \right).$$
But by Lemma \ref{lemma:epsreg}$b,$ we have
$$S(2)^2 \leq C' \left( \| \T(u_i) \|^2 + E_\p(u_i, U')\right),$$
where $C'$ depends on $\sigma$ and $U'.$
Together with (\ref{3int:smalltension}), this gives
$$E_\p(u_i, \hat{U})  \leq  C' \left( \frac{E_\p(u_i, \hat{U})}{i}  + E_\p(u_i, U') \right).$$
For $i$ sufficiently large, we may rearrange to obtain
$$E_\p(u_i, \hat{U})  \leq  C' E_\p(u_i, U'),$$
which by (\ref{3int:smalltension}), gives
\begin{equation*}
\| \T(u_i) \|_{L^2(\hat{U})}^2 \leq C' \frac{E_{\p}(u_i, U')}{i}.
\end{equation*}
The argument now proceeds by letting
$$\alpha_i = \frac{\p u_i}{S(2)}$$
and passing to a subsequence to obtain a holomorphic limit on $U',$ as in the proof of Lemma \ref{lemma:supinf}. One can then enlarge $U'$ and pass to a diagonal subsequence to obtain a holomorphic limit on $\hat{U},$ which contradicts the implication ($a'$). The proof of ($b$) is the same.
\end{proof}

\vspace{5mm}

\subsection{Estimates based on 3-annulus estimate}

\begin{prop}\label{prop:Edf} Given $L \in \N,$ there exists $C_L > 0$ as follows. For $0 < \beta \leq \frac12,$ $\sigma$ satisfying 
\begin{equation}\label{Edf:betassn}
\sigma^\beta \geq C_L,
\end{equation}
and $\sqrt[3]{2} \leq \tau \leq \sqrt{\sigma},$ 
there exists $\xi_0 > 0$ as follows.

Let $m,n \in \{ -L, \ldots, L \},$ $0 < \sigma^2 \rho \leq R \leq 1,$ $0 < \xi \leq \xi_0,$ and $0 \leq \beta \leq \frac12.$ Given $u : U^{2R}_{\rho/2} (x_0) \to N,$ let $\delta = \| \T (u) \|_{L^2 \left( U^{2R}_{\rho/2} \right) }.$ 
Suppose that
\begin{equation*}
\sup_{\rho / 2 \leq r \leq R} E(u, U_r^{2r}) < \xi_0.
\end{equation*}
Let $\mu, \nu$ be such that
\begin{equation*}
\mu \geq \max \left\{ \frac{ \sqrt{ E_\p \left(u,  U_{\rho/2}^{\rho}  \right)  } }{\rho} , \xi^{-1} \delta \right\}, \qquad
\nu \geq \max \left\{ \frac{ \sqrt{ E_\p \left(u, U_R^{2R}  \right) } }{R} , \xi^{-1} \delta \right\}.
\end{equation*}
Write
\begin{equation*}
f_\p(r) = \max \left\{ \sqrt{ \dashint_{S^1_r(x_0)} e_\p(u) \, d\theta}, \xi^{-1} \delta \right\}
\end{equation*}
and, for any $t > 1,$
$$F^\p_t(r) = \sup_{r \leq s \leq t r } \left( \frac{s}{\rho} \right)^{m} f_{\p} (s),$$
$$G^\p_t(r) = \sup_{t^{-1} r \leq s \leq r } \left( \frac{R}{s} \right)^{n} f_{\p}(s).$$
The following implications hold.

\vspace{2mm}

\noindent (a) If
\begin{equation}\label{Edf:aassumption}
\sigma^{1 - \beta} F^\p_{\tau}(\sigma \rho) \geq \mu 
\end{equation}
then
\begin{equation*}
C^{-1}_{m} \sigma^{ - 1}  \left( \frac{\rho}{r} \right)^{\beta} \mu \leq F^\p_\sigma (r) \leq C_m \left( \frac{R}{\sigma r} \right)^{\beta} \left( \frac{R}{\rho} \right)^m \nu 
\end{equation*}
for each $\sigma \tau \rho \leq r \leq R/\sigma.$

\vspace{2mm}

\noindent (b) If
\begin{equation}\label{Edf:bassumption}
\sigma^{1 - \beta} G^\p_\tau(R/\sigma) \geq \nu 
\end{equation}
then
\begin{equation*}
C^{-1}_{n} \sigma^{- 1} \left( \frac{ r}{R} \right)^{\beta} \nu \leq G^\p_\sigma(r) \leq C_n \left( \frac{ r}{ \sigma \rho } \right)^{\beta} \left( \frac{R}{\rho} \right)^n \mu. 
\end{equation*}
for each $\tau \rho \leq r \leq R / \sigma \tau .$

\vspace{2mm}

\noindent (c) If both (\ref{Edf:aassumption}) and (\ref{Edf:bassumption}) hold,
then
\begin{equation*}
\begin{split}
f_\p(r) \geq C_{L}^{-1} \left( \sup_{r/\sqrt{2} \leq s \leq \sqrt{2} r} f_\p(s) + \sigma^{-2} \mu \left( \frac{\rho}{r} \right)^{m + \beta} + \sigma^{- 2} \nu \left( \frac{r}{R} \right)^{n + \beta} \right)
\end{split}
\end{equation*}
for each $\rho \leq r \leq R.$

The same statements hold after replacing $e_\p$ by $e_\dbar$ or $e.$

\end{prop}
\begin{proof}

We shall assume that $\tau = \sqrt[3]{2}$ and $\sigma = \tau^M$ for some $M \in \N$ sufficiently large. The general case follows easily from this case.

We first claim that 
for $\xi_0 > 0$ sufficiently small, we have the implication 
\begin{equation}\label{Edf:implication}
F_{\tau}(\sigma^{-1} r) \leq \sigma^{1 - \beta} F_{\tau} (r) \quad \Rightarrow \quad F_{\tau} (r) \leq \sigma^\beta F_{\tau}(\sigma r).
\end{equation}
as long as $\frac{\sigma \rho}{2} \leq r \leq \frac{2R}{\sigma \tau}.$

Assume first that
\begin{equation}\label{Edf:firstcase}
\sup_{r \leq s \leq \tau r } \sqrt{ \dashint_{S^1_r(x_0)} e_\p(u) \, d\theta} \leq \xi^{-1} \delta.
\end{equation}
Supposing that $m \geq 0,$ we have $F_{\tau} (r) = \left( \dfrac{{\tau} r}{\rho} \right)^{m} \xi^{-1} \delta$ by definition. We also have
$$F_{\tau}({\sigma} r) \geq \left( \frac{\sigma \tau r}{\rho} \right)^{m} \xi^{-1} \delta = \sigma^m F_{\tau}(r) \geq F_{\tau}(r),$$
so the RHS of (\ref{Edf:implication}) is automatically true. Supposing that $m < 0,$ we have $F_{\tau} (r) = \left( \dfrac{r}{\rho} \right)^{m} \xi^{-1} \delta$ by definition. We also have
$$F_{\tau}({\sigma^{-1}} r) \geq \left( \frac{\sigma^{-1} r}{\rho} \right)^{m} \xi^{-1} \delta = \sigma^{-m} F_{\tau}(r) > \sigma^{1 - \beta} F_{\tau}(r),$$
since $-m > 1 - \beta,$ so
the LHS of (\ref{Edf:implication}) is automatically false.
This establishes the implication (\ref{Edf:implication}) in the case (\ref{Edf:firstcase}).

On the other hand, if (\ref{Edf:firstcase}) is not true, we have
$$\sup_{r \leq s \leq {\tau} r } \sqrt{ \dashint_{S^1_r(x_0)} e_\p(u) \, d\theta} > \xi^{-1} \delta.$$
In this case, we have
$$C_{m,{\tau} } \xi^2 E_\p(U_{r}^{ {\tau} r}) \geq r^2 \delta^2,$$
which guarantees the assumption (\ref{3int:tensionassumption}). 
Lemma \ref{lemma:3int}\ref{3int:1stimpl} now completes the proof of the implication (\ref{Edf:implication}).

Similarly, for $\xi_0$ sufficiently small, we can assume that the implication
\begin{equation}\label{Edf:tauimplication}
F_{\tau}(\tau^{-1} r) \leq \tau^{\frac12 } F_{\tau} (r) \quad \Rightarrow \quad F_{\tau} (r) \leq \tau^{\frac12} F_{\tau}(\tau r).
\end{equation}
is true.

To prove ($a$), first note that Lemma \ref{lemma:epsreg}$b$ implies
$$F^\p_\tau\left( \frac{\rho}{\tau^2} \right) \leq C_m \mu.$$
Combined with the the assumption (\ref{Edf:aassumption}), we have
$$F^\p_\tau\left( \frac{\rho}{\tau^2} \right) \leq C_m \sigma^{1 - \beta} F_{\tau}^\p(\sigma \rho).$$
We can replace $\beta$ by $\beta / 2$ to absorb the constant and gain a factor of $\tau^{2 - 2\beta}.$ 
This gives 
$$F^\p_\tau\left( \frac{\rho}{\tau^2} \right) \leq (\sigma \tau^2 )^{1 - \beta} F_{\tau}^\p\left( \sigma \rho \right).$$
We now have the LHS of the implication (\ref{Edf:implication}), where $r = \sigma \rho$ and $\bar{\sigma} = \sigma \tau^2$ plays the role of $\sigma.$ Also let $\bar{\rho} = \frac{\rho}{\tau^2}.$ 
Applying the implication repeatedly, we obtain
\begin{equation}\label{Fpartialii+1}
F^\p_\tau\left( \bar{\sigma}^{i} \bar{\rho} \right) \leq \bar{\sigma}^{\beta} F^\p_\tau(\bar{\sigma}^{i + 1} \bar{\rho})
\end{equation}
for all $i$ in the relevant range, and
$$F^\p_\tau(\bar{\sigma}^i \bar{\rho}) \geq \bar{\sigma}^{(i-1)\beta} F^\p_\tau(\bar{\sigma} \bar{\rho} ) \geq C_m^{-1} \mu.$$
These give the first inequality in (a) for the case that $r = \bar{\sigma}^i \bar{\rho}.$
Moreover, there must exist $i \in \{M, \ldots, 2M+1 \}$ (where $\sigma = \tau^M$) such that
\begin{equation}\label{Edf:'tauimplication}
F^\p_\tau\left( \tau^{i} \rho \right) \leq \tau^{1/2} F^\p_\tau(\tau^{i + 1} \rho),
\end{equation}
since otherwise the $i = 1$ case of (\ref{Fpartialii+1}) 
could not hold.
In view of (\ref{Edf:tauimplication}), the inequality (\ref{Edf:'tauimplication}) holds for all $i \geq 2M + 1$ in the relevant range. The same is true after putting $\sqrt{\sigma}$ in the role of $\tau.$ Drawing the same implications for $\tau$ and $\sqrt{\sigma}$ as we did for $\sigma,$ we have enough to obtain both implications of (a) for all $r$ in the stated range. The proof of (b) is similar.

To obtain the first estimate in ($c$), namely $f_\p(r) \geq C_{L}^{-1} \sup_{r/2 \leq s \leq 2r} f_\p(s),$ observe that (\ref{Edf:tauimplication}) and the corresponding statement for $G_\tau$ imply (\ref{supinf:assumption}). Then the desired bound can be derived from (\ref{supinf:supinf}). The last two estimates in ($c$) then follow from ($a$) and ($b$).
\end{proof}

\begin{cor}\label{cor:EdEdbarcomparable}
Let 
$u$ be as in the previous proposition. 
Let $\mu, \nu$ be such that
\begin{equation*}
\mu \geq \max \left\{ \frac{ \sqrt{ E \left(u,  U_{\rho/2}^{\rho}  \right)  } }{\rho} , \xi^{-1} \delta \right\}
\end{equation*}
and
\begin{equation*}
\nu \geq \max \left\{ \frac{ \sqrt{ E \left(u, U_R^{2R}  \right) } }{R} , \xi^{-1} \delta \right\}.
\end{equation*}
Write
\begin{equation*}
f_\p(r) = \max \left\{ \sqrt{ \dashint_{S^1_r(x_0)} e_\p(u) \, d\theta}, \xi^{-1} \delta \right\}, \qquad f_\dbar(r) = \max \left\{ \sqrt{ \dashint_{S^1_r(x_0)} e_\dbar(u) \, d\theta}, \xi^{-1} \delta \right\},
\end{equation*}
\begin{equation*}
f(r) = \max \left\{ \sqrt{ \dashint_{S^1_r(x_0)} e(u) \, d\theta}, \xi^{-1} \delta \right\},
\end{equation*}
$$F^\dbar_\tau(r) = \sup_{r \leq s \leq \tau r } \left( \frac{s}{\rho} \right)^{m} f_\dbar(s), \qquad G^\p_\tau(r) = \sup_{\tau^{-1} r \leq s \leq r } \left( \frac{R}{s} \right)^{n} f_\p(s).$$

\vspace{2mm}

\noindent If
\begin{equation}\label{Edbarcomparable:aassumption}
\mu \leq \sigma^{1 - \beta} F^\dbar_{\tau}(\sigma \rho) 
\end{equation}
and
\begin{equation}\label{Edbarcomparable:bassumption}
\nu \leq \sigma^{1 - \beta} G^\p_\tau(R/\sigma) 
\end{equation}
then 
there exists $\sigma \rho \leq \bar{\rho} \leq R/\sigma$ such that
\begin{equation*}
\begin{split}
\inf_{\bar{\rho} \leq r \leq 2 \bar{\rho}} f_\p(r) \wedge f_{\dbar}(r) \geq C_{\sigma}^{-1} \left( \sup_{\bar{\rho}/2\leq s \leq 4 \bar{\rho}} f(s) + \mu \left( \frac{\rho}{\bar{\rho}} \right)^{m + \beta} + \nu \left( \frac{\bar{\rho}}{R} \right)^{n + \beta} \right).
\end{split}
\end{equation*}
\end{cor}
\begin{proof}
The assumptions (\ref{Edbarcomparable:aassumption}-\ref{Edbarcomparable:bassumption}) imply that the full energy density $e(u)$ obeys the hypotheses of Proposition \ref{prop:Edf}($c$), so we conclude
\begin{equation*}
\begin{split}
\inf_{r/2 \leq s \leq 2 r} f(s) \geq C_{L}^{-1} \left( f(r) + \sigma^{-2} \mu \left( \frac{\rho}{\bar{\rho}} \right)^{m + \beta} + \sigma^{-2} \nu \left( \frac{\bar{\rho}}{R} \right)^{n + \beta} \right).
\end{split}
\end{equation*}
We can also apply Proposition \ref{prop:Edf}$a$ to $e_\p$ to learn that
$$f_\p(r_1) \geq C_\sigma^{-1} f(r_1)$$
for some $\sigma \rho \leq r_1 \leq 2 \sigma \rho.$ Applying Proposition \ref{prop:Edf}$b$ to $e_{\dbar},$ we learn that
$$f_{\dbar}(r_2) \geq C_\sigma^{-1} f(r_2)$$
for some $R/ 2 \sigma \leq r_2 \leq R / \sigma.$
Using the fact that $e(u) = e_\p(u) + e_{\dbar}(u)$ and the Intermediate Value Theorem, we obtain the desired conclusion.
\end{proof}

\vspace{2mm}



\subsection{Multi-annulus estimates}\label{ss:multiannulus}
We now give a generalization that is needed in order to pass the lower bounds across ghost bubbles. This relies on the following

\begin{lemma}\label{lemma:multiannulus} 

Let $0 < \beta \leq \frac12$ and
$n_i \in \Z, $ $i = 1, \ldots, \ell \leq k,$ with $|n_i| \leq L,$ and put
\begin{equation}\label{mdef}
m = -\sum_{i = 1}^\ell n_i.
\end{equation}
Let $0 < \zeta \leq \frac12.$ There exists $\sigma_0 > 1,$ depending on $k, L$ and $\zeta,$ such that for all $\sigma \geq \sigma_0,$ the following holds.

Let $z_i, z_i' \in B_{1/2}(0) \subset \C,$ $\zeta_i \in \LB \zeta, \frac12 \RB,$ and $\zeta'_i \in \left( 0, \zeta_i / \sigma \RB,$
for $i = 1, \ldots, \ell.$ Assume that 
$$B_{2 \zeta_i }(z_i) \cap B_{2 \zeta_j }(z_j) = \emptyset,$$
for $1 \leq i \neq j \leq \ell,$ and
$$B_{\zeta'_i }(z'_i) \subset B_{\zeta_i / \sigma }(z_i)$$
for $i = 1, \ldots, \ell.$
Write
\begin{equation*}
\begin{split}
\Lambda & = B_{1}(0) \setminus \cup_{i = 1}^\ell \bar{B}_{\zeta_i }(z_i), \\
\hat{\Lambda} & = B_{\sigma}(0) \setminus \cup_{i = 1}^\ell \bar{B}_{\zeta'_i}(z'_i).
\end{split}
\end{equation*}
Suppose that $f(z)$ is a nonzero holomorphic function on $\hat{\Lambda} \subset \C.$
Let
$$S_i(1) = \sqrt{ \left( \zeta'_i \right)^{-2n_i} \dashint_{S^1_{ \zeta'_i }\left( z'_i \right)} |f(z)|^2  \, d\theta }$$
for $i = 1, \ldots, \ell,$
$$S(2) = \sqrt{\dashint_\Lambda |f(z)|^2\, dz},$$
and
$$S(3) = \sqrt{  \sigma^{2 m} \dashint_{S^1_\sigma (0) } |f(z)|^2  \, d\theta }.$$
The following implication holds:
\begin{equation*}
\begin{split}
\sum_{i = 1}^\ell S_i(1)  \leq \sigma^{1 - \beta} S(2) \quad \Rightarrow \quad S(2) < \sigma^\beta S(3).
\end{split}
\end{equation*}
\end{lemma}

\begin{proof}  Dividing $f(z)$ by a rational function with a pole of order $n_i$ at $z_i,$ we can reduce to the case $m = 0 = n_i$ for all $i.$ This only affects the estimates up to constants depending on $L,k,$ and $\zeta,$ which can be absorbed into $\sigma_0.$

We now use the following ``multiple'' Laurent expansion: 
\begin{equation}
\begin{split}
f(z) & = \frac{1}{2 \pi i} \int_{ |w| = \sigma } \frac{f(w)}{w - z} dw - \sum_{i = 1}^\ell \frac{1}{2 \pi i} \int_{ |w - z'_i| = \zeta'_i} \frac{f(w)}{w - z} dw \\
& = \sum_{n = 0}^\infty a_n z^n + \sum_{i = 1}^\ell \sum_{n = 1}^\infty b_{n,i} (z - z'_i)^{-n} \\
& =: f_0(z) + \sum_{i = 1}^\ell f_i(z).
\end{split}
\end{equation}
The expansion is obtained by inserting the identity $\frac{1}{w - z} = \sum_{n = 0}^\infty \frac{z^n}{w^{n + 1}}$ 
in the first term and $\frac{1}{w - z} = \sum_{i = 0}^\infty \frac{-(w - z'_i)^n}{(z - z'_i)^{n + 1}}$ 
in the remaining terms. 
By construction, $a_1, a_2, \ldots$ are the nonnegative Fourier coefficients of $f(z)$ along the circle $|z| = \sigma$ and $b_{1,i}, b_{2,i}, \ldots$ are the negative Fourier coefficients of $f(z)$ along $|z - z'_i| = \zeta'_i.$
We therefore have the bounds
\begin{equation*}
\sum_{n = 0}^\infty |a_n|^2 \sigma^{2n} \leq S(3)^2
\end{equation*}
and
\begin{equation*}
\sum_{n = 1}^\infty |b_{n,i}|^2 \left( \zeta_i' \right)^{-2n} = (\zeta_i' )^{-2} \sum_{n = 1}^\infty |b_{n,i}|^2 \left( \zeta_i' \right)^{2-2n} \leq 
S_i(1)^2
\end{equation*}
for $i = 1, \ldots, \ell,$ which gives
\begin{equation*}
\sum_{n = 1}^\infty |b_{n,i}|^2 \left( \frac{\sigma}{\zeta_i} \right)^{2n-2} \leq  \frac{\zeta_i^2}{\sigma^2} S_i(1)^2.
\end{equation*}
We have
\begin{equation*}
\dashint_{\Lambda} |f_0(z)|^2 \, dV_z \leq \frac{2}{\pi} \int_{B_{1}} |f_0(z)|^2 \, dV_z = \frac{2}{\pi} \sum |a_n|^2 \int_{B_{1}} |z|^{2n} \, dV_z = 4 \sum_{n = 0}^\infty \frac{ |a_n|^2}{2n + 2} \leq 2 S(3)^2 
\end{equation*}
and, for $i = 1, \ldots, \ell,$
\begin{equation*}
\begin{split}
\dashint_{\Lambda} |f_i(z)|^2 \, dV_z \leq \frac{2}{\pi} \int_{U^{1}_{\zeta_i} (z_i)} |f_i(z)|^2 \, dV_z
& = \frac{2}{\pi} \sum_{n = 1}^\infty |b_{n,i}|^2 \int_{ U^{1}_{\zeta_i} (z_i) } |z - z'_i |^{-2n} \, dV_z \\
& \leq 4 \left( |b_{1,i}|^2 (-\log \zeta_i) + \sum_{n = 2}^\infty \frac{|b_{n,i}|^2 }{ \zeta_i^{2n - 2}(2n - 2) } \right) \\
& \leq \frac{4 \zeta_i^2  }{\sigma^2} S_i(1)^2 \left( - \log \zeta_i + \frac{1}{\sigma^2}\right) \\
& \leq 2 \frac{S_i(1)^2}{\sigma^2}.
\end{split}
\end{equation*}
Putting these together, we have
\[
\begin{split}
S(2) = \sqrt{ \int_{\Lambda} |f(z)|^2 \, dV_z } & \leq \sqrt{ \int_{\Lambda} |f_0(z)|^2 \, dV_z } + \sum_{i = 1}^\ell \sqrt{ \int_{\Lambda} |f_i(z)|^2 \, dV_z } \\
& \leq  \sqrt{2} \left( S(3) + \frac{ \sum_{i = 1}^\ell S_i }{\sigma} \right).
\end{split}
\]
This clearly implies the claimed inequality.
\end{proof}

Recall that $D_r(x) \subset S^2$ denotes the image of the ball of radius $r$ in the stereographic chart centered at $x.$ We shall also write $S^1_r(x) = \p D_r(x).$

\begin{thm}\label{thm:puncturedsphere} Suppose again that $N$ is K\"ahler. Let $k, L \in \N$ and $0 < \beta, \zeta \leq \frac12.$
There exist $\sigma_0 >1,$ depending on $k, L, \beta,$ and $\zeta,$
such that assuming $\sigma \geq \sigma_0,$ 
there exists $\xi_1 > 0,$ depending 
also on $N$ and $\sigma,$ as follows.

Fix
$n_i \in \Z, $ $i = 0, \ldots, \ell \leq k,$ with $|n_i| \leq L$ and 
\begin{equation}\label{sumniis2}
\sum_{i = 0}^\ell n_i = -2.
\end{equation}
Let $x_i \in S^2$ for $i = 0, \ldots, \ell,$ and $\zeta_i \in \LB \zeta, \frac12 \RB$ for $i = 0, \ldots, \ell.$ Assume that 
$$D_{2\zeta_i}(x_i) \cap D_{2 \zeta_j }(x_j) = \emptyset,$$
for $0 \leq i \neq j \leq \ell.$ Write
\begin{equation*}
\begin{split}
\Lambda & = S^2 \setminus \bigcup_{i = 0}^\ell \bar{D}_{\zeta_i }(x_i), \\
\hat{\Lambda} & = S^2 \setminus \bigcup_{i = 0}^\ell \bar{D}_{\zeta_i / \sigma }(x_i).
\end{split}
\end{equation*}
Suppose that $u : \hat{\Lambda} \to N$ is a nonconstant $W^{2,2}$ map with
\begin{equation*}
E(u,\hat{\Lambda} ) < \xi_1
\end{equation*}
and
\begin{equation*}
 \| \T(u) \|^2_{L^2(\hat{\Lambda})} < \xi_1 E_\p(u, \hat{\Lambda}) .
\end{equation*}
Let
$$S_i = \sqrt{ \sup_{ \frac{\zeta_i }{\sigma } < r \leq \frac{ 2 \zeta_i }{ \sigma } } r ^{-2n_i} \dashint_{S^1_r(x_i)} e_\p (u)  \, d\theta }$$
for $i = 0, \ldots, \ell.$ 
The following implication holds:
\begin{equation}\label{puncturedsphere:implication}
\begin{split}
\sum_{i = 1}^\ell S_i \leq \sigma^{1 - \beta} E_\p (u, \Lambda) \Rightarrow  E_\p (u, \Lambda) < \sigma^{\beta} S_0.
\end{split}
\end{equation}
The same result holds after replacing $e_\p$ by $e_{\dbar}$ or $e.$
\end{thm}

\begin{proof}
By a compactness-and-contradiction argument as in the proofs of Lemmas \ref{lemma:supinf}-\ref{lemma:3int}, we can reduce to checking the case of a holomorphic 1-form $\alpha$ on $\hat{\Lambda}.$

Working in a stereographic chart centered at the antipodal point $\hat{x}_0$ of $x_0,$ we can write
$$\alpha = f(z) \, dz$$
for a holomorphic function $f(z).$ 
We rescale the chart so that the circle $S^1_{\zeta_0}(x_0)$ corresponds to the unit circle in $\C.$
The circle $S^1_{\zeta_0/\sigma }(x_0) \subset S^2$ goes over to the circle $\{ |z| = \sigma \} \subset \C,$ and for $i \geq 1,$ the circles $S^1_{\zeta_i / \sigma} (x_i) $ go over to circles $\{|z - z'_i| = \zeta_i' \} \subset \C$ for some points $z'_i \in D_{1/2}(0) \subset \C$ and radii $\zeta_i'.$ Since $|dz| = (1 + |z|^2)/\sqrt{2},$ we can take $m = n_0 + 2$ in the previous Lemma and obtain the desired result.
\end{proof}

\begin{cor}\label{cor:multiannulus} 

Without assuming (\ref{sumniis2}), fix
$n_i \in \Z, $ $i = 1, \ldots, \ell,$ with $|n_i| \leq L,$ and define $m$ by (\ref{mdef}).
Let $0 < \rho \leq 1,$ $x_0 \in S^2,$ and $x_i \in D_{\rho/2}(\hat{x}_0),$ for $i = 1, \ldots, \ell,$ with $\ell \leq k.$ Assume that 
$$D_{2 \rho \zeta}(x_i) \cap D_{2 \rho \zeta}(x_j) = \emptyset,$$
for $1 \leq i \neq j \leq \ell.$ 
Write
\begin{equation*}
\begin{split}
\Lambda' & = D_{\rho}(\hat{x}_0) \setminus \cup_i \bar{D}_{\rho \zeta }(x_i), \\
\hat{\Lambda}' & = D_{\rho \sigma}(\hat{x}_0) \setminus \cup_i \bar{D}_{\rho \zeta / \sigma }(x_i).
\end{split}
\end{equation*}
Suppose that $u : \hat{\Lambda}' \to N$ is a nonconstant $W^{2,2}$ map with
\begin{equation*}
E(u,\hat{\Lambda}' ) < \xi_1
\end{equation*}
and
\begin{equation}\label{multiannulus:corThypothesis}
\rho^2 \| \T(u) \|^2_{L^2(\hat{\Lambda}')} \leq \xi_1 E_\p(u, \hat{\Lambda}') .
\end{equation}
Let
$$S_i(1) = \sqrt{ \sup_{\frac{\rho\zeta}{\sigma} \leq r \leq \frac{2\rho \zeta}{ \sigma} } \left( \frac{r}{\rho} \right)^{-2n_i} \dashint_{S^1_r(x_i)} e_\p (u)  \, d\theta }$$
for $i = 1, \ldots, \ell,$ and
$$S(2) = \frac{\sqrt{E_\p (u, \Lambda')} }{\rho}.$$
Let
$$S(3) = \sqrt{ \sup_{ \frac{\rho \sigma}{ 2 } \leq r \leq \rho \sigma } \left( \frac{r}{\rho} \right)^{2 m} \left( 1 + r^2\right)^{-2} \dashint_{S^1_r(\hat{x}_0) } e_\p (u)  \, d\theta }.$$
The following implications hold:
\begin{equation}\tag{$a$}
\begin{split}
\sum_{i = 1}^\ell S_i(1)  \leq \sigma^{1 - \beta} S(2) \quad \Rightarrow \quad S(2) < \sigma^\beta S(3).
\end{split}
\end{equation}
\begin{equation}\tag{$b$}
\begin{split}
\sum_{i \neq j} S_i(1) + S(3) \leq \sigma^{1-\beta} S(2) \quad \Rightarrow \quad  S(2) < \sigma^\beta S_j(1).
\end{split}
\end{equation}
The same result holds after replacing $e_\p$ by $e_{\dbar}$ or $e.$
\end{cor}

\begin{proof} These both follow from the previous theorem (with $n_0 = m - 2$) after rescaling so that $\rho = 1.$
\end{proof}

\vspace{10mm}

\section{Bubble-tree estimates}


\begin{defn}
A sequence $u_i : \Sigma \to  N$ of $W^{2,2}$ maps with $E(u_i) \leq 4 \pi k$ and
$\|\T(u_i)\|_{L^2} \to 0$
as $i \to \infty$ will be called an \emph{almost-harmonic sequence}.
\end{defn}

\begin{defn}\label{defn:bubbletreedata} A collection
$$\mathcal{B} = \{ J,\{z_j \}, \{y_{j_1, \ldots, j_q} \}, \{ \phi_{j_1, \ldots, j_q } \}, u_\infty \}$$
will be referred to as \emph{bubble-tree data}. Here, $J = \{(j_1, \ldots, j_q ) \}$ 
is a finite indexing set; $\{z_{j} \} \subset \Sigma,$ for $(j) \in J,$ is a finite set of mutually distinct points; $y_{j_1, \ldots, j_q} \in B_{1/2}(0) \subset \R^2,$ for $(j_1, \ldots, j_q) \in J$ ($q \geq 2$), are finite sets of points such that $y_{j_1, \ldots, j_q,k} \neq y_{j_1, \ldots, j_q, \ell}$ for $k \neq \ell$;
and
\begin{equation*}
u_\infty = \phi_0 : \Sigma \to N,
\end{equation*}
\begin{equation*}
\phi_{j_1, \ldots, j_q} : \R^2 \to N,
\end{equation*}
for $(j_1, \ldots, j_q) \in J,$ are finite-energy harmonic maps, some of which may be constant. If $(j_1, \ldots, j_q) \in J$ is a terminal index ({\it i.e.} $(j_1, \ldots, j_q, j) \not\in J$ for all $j$), then $\phi_{j_1, \ldots, j_q}$ is required to be nonconstant.

Furthermore, $J$ is required to have a ``tree'' structure rooted at $(0)\in J,$ i.e., if $\vec{\jmath} = (j_1, \ldots, j_q) \in J$ then $\vec{\jmath}' = (j_1, \ldots, j_p) \in J$ for $p = 1, \ldots, q-1.$ We say that an index $\vec{\jmath}' \in J$ of this form ``precedes'' $\vec{\jmath}  \in J,$ and $\vec{\jmath}$ ``succeeds'' $\vec{\jmath}'.$ The initial (zero-length) index is denoted $(0).$

\end{defn}

\begin{defn}[\cite{songwaldron}, Definition 4.2] Given a $W^{1,2}$ map $u: B_R(x_0) \to N,$ the \emph{outer energy scale} $\lambda_{R, \eps, x_0}(u)$ is the smallest number $\lambda \geq 0$ such that
\begin{equation}\label{energyscale}
\sup_{\lambda < \rho < R} E \left( u, U^{\rho}_{\rho/2}(x_0) \right) < \eps.
\end{equation}
Note that $\lambda = R$ satisfies (\ref{energyscale}) vacuously, so $0 \leq \lambda_{R,\eps, x_0}(u) \leq R$ by definition.
\end{defn}

\begin{thm}[Bubble-tree compactness for almost-harmonic sequences]\label{thm:bubbletree}
Given an almost-harmonic sequence $u_i : \Sigma \to N,$ we may pass to a subsequence (again called $u_i$) which ``converges in the bubble-tree sense,'' i.e., for which there exists a set of bubble-tree data $\mathcal{B}$ as in Definition \ref{defn:bubbletreedata},
as follows.

Given any $0 < \eps \leq \eps_0,$ there exists $0 < \zeta \leq \frac12$ and for all $i$ sufficiently large, 
points $x^i_{j_1, \ldots, j_q} \in \Sigma$ and positive numbers $\lambda^i_{j_1, \ldots, j_q},$ such that:

\begin{itemize}

\item $|J| \leq L,$ where $L \in \N$ depends only on $k,$ 

\item $x^i_{j_1, \ldots, j_q} \to z_{j_1}$ as $i \to \infty,$ 

\item $\lambda^i_{j_1} = \lambda_{\zeta, \eps, x^i_{j_1} }(u_i) \to 0 $ as $i \to \infty$

\item $\lambda^i_{j_1, \ldots, j_q} = \lambda_{\zeta \lambda^i_{j_1, \ldots, j_{q - 1}}, \eps, x^i_{j_1, \ldots, j_q} }(u_i),$ 


\item $x^i_{j_1, \ldots, j_q} = x^i_{j_1, \ldots, j_{q - 1}} + \lambda^i_{j_1, \ldots, j_{q-1}} y_{j_1, \ldots, j_q},$

\item $\dfrac{\lambda^i_{j_1, \ldots, j_q}}{ \lambda^i_{j_1, \ldots, j_{q - 1}} } \to 0$ as $i \to \infty,$


\item $u_i \left( x^i_{ j_1, \ldots, j_q} + \lambda^i_{j_1, \ldots, j_q} y\right) \to \phi_{j_1, \ldots, j_q}(y) $ in $W^{2,2}_{loc} \left( \R^2 \setminus \cup_{k} y_{j_1, \ldots, j_{q}, k} \right)$ as $i \to \infty,$

\item For all $i$ sufficiently large, we have $E \left(u_i, U_{\lambda^i_{j_1, \ldots, j_q}}^{\zeta \lambda^i_{j_1, \ldots, j_{q - 1}} }(x^i_{j_1, \ldots, j_q}) \right) \leq C \eps$ for a universal constant $C.$


\end{itemize}
\end{thm}
\begin{proof} The proof is by now standard (see Song-Waldron \cite{songwaldron} for a recent sketch), and we omit it.
\end{proof}

\begin{center}

\includegraphics[scale=.4]{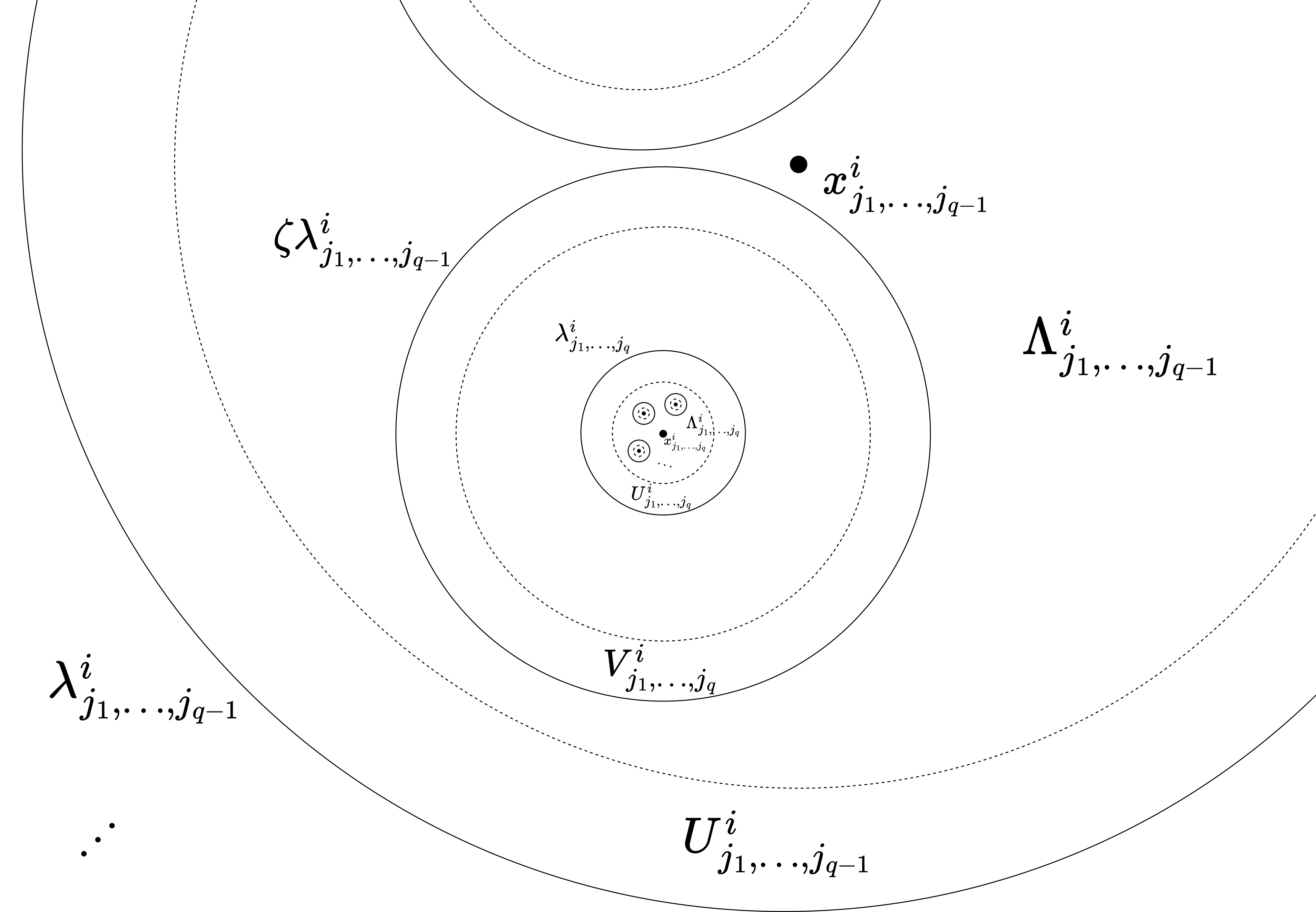}

\end{center}


\begin{defn}\label{defn:jandJ} Given a set of bubble-tree data $\mathcal{B},$ write
$$J = J_0 \sqcup J_1,$$
where $J_0$ are the constant (a.k.a. ghost) indices and $J_1$ are the nonconstant indices. Define two maps, $\cev{\jmath}(\cdot)$ and $\vec{J}(\cdot),$ as follows. Given any index $\vec{\jmath} = (j_1, \ldots, j_q),$ write
$$\cev{\jmath}(\vec{\jmath }) = ( j_1, \ldots, j_p) \in J_1 \cup \{(0)\},$$
with $p \leq q,$ for the nearest nonconstant index preceding $\vec{\jmath}$ in the bubble tree, or $\cev{\jmath}(\vec{\jmath }) = (0)$ if there are only ghost (i.e. constant) bubbles between $\vec{\jmath}$ and the body map.
Write
$$\vec{J}(\vec{\jmath}) = \{ (j_1, \ldots, j_q, \ldots, j_r) \} \subset J_1$$
for the set of nonconstant indices succeeding $\vec{\jmath}$ in the bubble tree, i.e., $(j_1, \ldots, j_q, \ldots, j_r) \in \vec{J}(\vec{\jmath})$ iff $(j_1, \ldots, j_q, \ldots, j_r) \in J_1$ but $(j_1, \ldots, j_p, \ldots, j_s) \in J_0$ for any $p \leq s < r.$ 
\end{defn}

\begin{defn}\label{defn:Mexponents}
Suppose that $\phi_{j_1, \ldots, j_q}$ and $u_\infty$ are each either holomorphic, antiholomorphic, or constant, as is the case for $\Sigma = N = S^2.$ For each nonzero index $\vec{\jmath} = (j_1, \ldots, j_q) \in J,$ define an integer $M_{j_1, \ldots, j_q} \geq 2$ as follows. For $\vec{\jmath} = (j_1, \ldots, j_q) \in J_1,$ 
let
$$M_{j_1, \ldots, j_q} = \text{order-of-vanishing of $d \phi_{j_1, \ldots, j_q}$ at $\infty$} .$$
For a ghost index $\vec{\jmath} \in J_0,$ define
\begin{equation*}
M_{j_1, \ldots, j_q} = \sum M_{j_1, \ldots, j_q, \ell_{q + 1}, \ldots, \ell_n},
\end{equation*}
where the sum runs over 
$(j_1, \ldots, j_q, \ell_{q + 1}, \ldots, \ell_n) \in \vec{J} (\vec{\jmath}).$ Note that this last definition is made so that (\ref{lowerboundsone:Mghostdef}) below will hold.
\end{defn}
\begin{defn}\label{defn:Nexponents}
Define nonnegative integers $N_{j_1, \ldots, j_q}$ as follows. Supposing first that
$$(j_1, \ldots, j_{q - 1}) \in J_1,$$
let
$$N_{j_1, \ldots, j_q} = \text{order-of-vanishing of $d \phi_{j_1, \ldots, j_{q-1}}$ at $y_{j_1, \ldots, j_q}$}.$$
Suppose next that the preceding index $(j_1, \ldots, j_{q - 1}) \in J_0$ is a ghost index. Assuming $q = 1$ (i.e. the body map $\phi_0$ is constant), define
$$N_{j_1} = \begin{cases} 0 & \text{ if only } (j_1) \in J \\ \left( \sum_{j \neq j_1} M_{j} \right) - 2 & \text{ otherwise.} \end{cases}$$
Assuming $q > 1,$ write 
$\cev{\jmath}(\vec{\jmath}) = (j_1, \ldots, j_p),$ and define
\begin{equation*}
N_{j_1, \ldots, j_q} = N_{j_1, \ldots, j_p}  + \sum_{n = p + 1}^q \sum_{j \neq j_n} M_{j_1, \ldots, j_p, \ldots, j_{n-1}, j}.
\end{equation*}
Note that this last definition is made so that (\ref{lowerboundsone:Nghostdef}) below will hold.
\end{defn}
\begin{rmk}\label{rmk:L}
The integers $M_{j_1, \ldots, j_q}$ and $N_{j_1, \ldots, j_q}$ are evidently bounded above by an integer depending only on $k$ and $|J|.$ We denote this integer again by $L,$ so that
$$M_{\vec{\jmath}}, N_{\vec{\jmath}} \leq L$$
for all $\vec{\jmath} \in J.$
Also note that the precise values of these exponents are not important for our main results and are provided here mainly for the sake of concreteness.
\end{rmk}

Now, given $0 < \eps \leq \eps_0,$ let
$$x^i_{j_1, \ldots, j_q} \in S^2, \qquad 0 < \lambda^i_{j_1, \ldots, j_q} \leq \zeta,$$
for $i = 1, 2, \ldots,$ be the sequences of points and scales guaranteed by Theorem \ref{thm:bubbletree}. For $q \geq 1,$ also let 
$$\Lambda^i_{j_1, \ldots, j_q} = D_{\lambda^i_{j_1, \ldots, j_q} } \left(x^i_{j_1, \ldots, j_q} \right) \setminus \cup_k \bar{D}_{\zeta \lambda^i_{j_1, \ldots, j_q} }\left(x^i_{j_1, \ldots, j_q, k} \right),$$
\begin{equation*}
U^i_{j_1, \ldots, j_q} = U^{\lambda^i_{j_1, \ldots, j_q} }_{\lambda^i_{j_1, \ldots, j_q} / 2}\left(x^i_{j_1, \ldots, j_q} \right) \subset \Lambda^i_{j_1, \ldots, j_q},
\end{equation*}
and
\begin{equation*}
V^i_{j_1, \ldots, j_{q } } = U^{2\zeta \lambda^i_{j_1, \ldots, j_{q -1}} }_{\zeta \lambda^i_{j_1, \ldots, j_{q - 1}}} \left( x^i_{j_1, \ldots, j_{q} } \right) \subset \Lambda^i_{j_1, \ldots, j_{q-1}}.
\end{equation*}
In addition, let
\begin{equation*}
\Lambda^i_0 = S^2 \setminus \cup_j \bar{D}_{\zeta} \left( x^i_j \right).
\end{equation*}
Given $\xi_1 > 0,$ for $i$ large enough, we may assume that for each ghost index $(j_1, \ldots, j_q) \in J_0,$ we have
\begin{equation}\label{lowpartial}
E \left(u_i, \Lambda^i_{j_1, \ldots, j_q} \right) < \xi_1.
\end{equation}
We are finally ready to prove our main bubble-tree estimate.

\begin{thm}\label{thm:lowerboundsone} Let $0 < \beta < \frac12.$ Suppose that $u_i : S^2 \to N,$ K\"ahler, is a bubble-tree convergent almost-harmonic sequence, with bubble-tree data $\mathcal{B},$ for which the body map and all bubbles are either holomorphic or antiholomorphic.\footnote{These assumptions can likely be removed if the exponents $M$ and $N$ are allowed to be arbitrary.} Write
$$\delta_i := \| \T(u_i) \| \to 0.$$
Given $0 < \eps \leq \eps_0,$ 
there exist $\kappa, \xi > 0,$ depending on $\eps,$ $\beta,$ and $\mathcal{B},$ as follows.

Let $\zeta,$ $\Lambda^i_{\vec{\jmath}},$ 
$M_{\vec{\jmath}},$ $N_{\vec{\jmath} },$ etc., be as above, 
and put
$$\mu^i_{\vec{\jmath} } = \max \left \{ \frac{ \sqrt{ \min \left\{ E \left( u_i, \Lambda^i_{\vec{\jmath} } \right), \eps \right\} } }{ \lambda^i_{\vec{\jmath} } }, \xi^{-1} \delta_i \right\}$$
and
\begin{equation}\label{lowerboundsone:fidef}
f^i_{\vec{\jmath} }(r) = \max \left\{ \sqrt{ \dashint_{S^1_r(x^i_{\vec{\jmath}} )} e(u_i) \, d\theta}, \xi^{-1} \delta_i \right\}
\end{equation}
for $\lambda^i_{j_1, \ldots, j_q} \leq r \leq \zeta \lambda^i_{j_1, \ldots, j_{q - 1}}$ (where $\vec{\jmath} = (j_1, \ldots, j_q)$).

\vspace{2mm}

\noindent ($a$) For $i$ sufficiently large, we have
\begin{equation}\label{lowerboundsone:aestimate}
f^i_{j_1, \ldots, j_q}(r) \geq \kappa \left( \mu^i_{j_1, \ldots, j_q} \left( \frac{\lambda^i_{j_1, \ldots, j_q} }{r} \right)^{M_{j_1, \ldots, j_{q} } + \beta } + \mu^i_{j_1, \ldots,j_{q - 1} } \left( \frac{r}{\lambda^i_{j_1, \ldots, j_{q - 1} } } \right)^{N_{j_1, \ldots, j_q} + \beta } \right)
\end{equation}
for each $\vec{\jmath} \in J$ and $\lambda^i_{j_1, \ldots, j_q} \leq r \leq \zeta \lambda^i_{j_1, \ldots, j_{q - 1}},$ as well as
\begin{equation}\label{lowerboundsone:bestimate}
f^i_{j_1, \ldots, j_q}(r) \geq c_{L} \sup_{r/2\leq s \leq 2 r} f^i_{j_1, \ldots, j_q}(s)
\end{equation}
for $\kappa^{-1} \lambda^i_{j_1, \ldots, j_q} \leq r \leq \kappa \zeta \lambda^i_{j_1, \ldots, j_{q - 1}}.$

\vspace{2mm}

\noindent ($b$) Suppose that 
\begin{equation}\label{lowerboundsone:ehypothesis}
\begin{split}
E_\p \left( u_i, \Lambda^i_{j_1, \ldots, j_q} \right) + (\lambda^i_{j_1, \ldots, j_q} \delta_i)^2 & \geq \xi E \left( u_i, \Lambda^i_{j_1, \ldots, j_q} \right) \,\, \text{ and }\\
E_\dbar \left( u_i, \Lambda^i_{j_1, \ldots, j_{q-1}} \right) + (\lambda^i_{j_1, \ldots, j_{q-1}} \delta_i)^2 & \geq \xi E \left( u_i, \Lambda^i_{j_1, \ldots, j_{q-1}} \right).
\end{split}
\end{equation}
Then 
there exists $\rho_i \in \LB 2\lambda^i_{j_1, \ldots, j_q}, \frac12 \zeta \lambda^i_{j_1, \ldots, j_{q-1} } \RB $ such that
\begin{equation}\label{lowerboundsone:eestimate}
\begin{split}
\inf_{\rho_i \leq r \leq 2 \rho_i } f^i_{\p, \vec{\jmath}} (r) \wedge f^i_{\dbar, \vec{\jmath} }(r) \geq \kappa \sup_{\rho_i/2\leq s \leq 4 \rho} f^i_{ \vec{\jmath}} (s).
\end{split}
\end{equation}
A similar statement holds with $\p$ and $\dbar$ interchanged.

\end{thm}
\begin{proof}  



\begin{claim}$\mathit{1}.$ We can choose $\sigma$ large enough that for each $\vec{\jmath} \in J$ and $i$ sufficiently large, there holds
\begin{equation}\label{lowerboundsone:claim1}
\mu^i_{\vec{\jmath}} \leq \sigma^{1 - \beta} \sup_{\frac12 \sigma \lambda^i_{\vec{\jmath}} \leq r \leq \sigma \lambda^i_{\vec{\jmath}} } \left( \frac{r}{\lambda^i_{\vec{\jmath}} } \right)^{M_{\vec{\jmath}}} f^i_{\vec{\jmath}} (r).
\end{equation}
\end{claim}
\begin{claimproof} First assume that $\vec{\jmath} \in J_1$ is a nonconstant index. Since $u_i$ is tending to $\phi_{\vec{\jmath}}$ in $W_{loc}^{2,2}$ in a rescaled sense, (\ref{lowerboundsone:claim1}) follows from the definition of $M_{\vec{\jmath }}$ 
for $\sigma$ and $i$ sufficiently large (with $\sigma$ depending only on $\mathcal{B}$).


We can now prove the claim for the ghost indices $J_0$ by ``inward'' induction on the bubble tree, the base case being that of the terminal bubbles (which all belong to $J_1$). Let $\vec{\jmath} = (j_1, \ldots, j_q) \in J_0.$ Note from Definition \ref{defn:Mexponents} that
\begin{equation}\label{lowerboundsone:Mghostdef}
M_{j_1, \ldots, j_q} = \sum_k M_{j_1, \ldots, j_q, k}.
\end{equation}
Assume that the claim has been proved for all indices succeeding $\vec{\jmath};$ in particular
\begin{equation*}
\mu^i_{j_1, \ldots, j_q, k} \leq \sigma^{1 - \beta} \sup_{\frac12 \sigma \lambda^i_{j_1, \ldots, j_{q},k } \leq r \leq \sigma \lambda^i_{j_1, \ldots, j_{q},k } } \left( \frac{r}{\lambda^i_{j_1, \ldots, j_{q},k } } \right)^{M_{j_1, \ldots, j_{q},k} } f^i_{j_1, \ldots, j_{q},k } (r)
\end{equation*}
for each $k$ such that $(j_1, \ldots, j_{q},k) \in J.$  
The second inequality of Proposition \ref{prop:Edf}$a$ then gives
\begin{equation}\label{lowerboundsone:Mlowerbound}
\begin{split}
\sup_{ \frac{ \zeta \lambda^i_{j_1, \ldots, j_{q}} }{\sigma } \leq r \leq \zeta \lambda^i_{j_1, \ldots, j_{q} } } \left( \frac{r}{ \zeta \lambda^i_{j_1, \ldots, j_{q} } } \right)^{M_{j_1, \ldots, j_{q},k} + \beta} f^i_{j_1, \ldots, j_{q},k } (r) & \leq \frac{ \sqrt{ E \left(u,  V_{j_1, \ldots, j_{q}, k }  \right)  } }{ \lambda^i_{j_1, \ldots, j_{q}} } \vee \xi^{-1} \delta_i \\
& \leq \mu^i_{j_1, \ldots, j_q}.
\end{split}
\end{equation}
First notice that if $\mu^i_{j_1, \ldots, j_q} = \xi^{-1} \delta_i,$ then the claim (\ref{lowerboundsone:claim1}) is trivial since $M_{j_1, \ldots, j_q} \geq 0.$ Otherwise, we have $\lambda^i_{\vec{\jmath} } \delta_i < \xi \sqrt{ E \left( u_i, \Lambda^i_{\vec{\jmath} } \right) },$ which guarantees (\ref{multiannulus:corThypothesis}).
We may apply Corollary \ref{cor:multiannulus}$a,$ with $m = M^\p_{j_1, \ldots, j_q}$ and $n_k = -M^\p_{j_1, \ldots, j_q, k}$ (in view of (\ref{lowerboundsone:Mghostdef}) and (\ref{lowerboundsone:Mlowerbound})), to complete the induction step.
\end{claimproof}


\begin{claim}$\mathit{2}.$ For each $\vec{\jmath} = (j_1, \ldots, j_q) \in J$ and $i$ sufficiently large, we have
\begin{equation}\label{lowerboundsone:claim2}
\mu^i_{j_1, \ldots, j_{q-1}} \leq \sigma^{1 - \beta} \sup_{ \frac{\zeta \lambda^i_{ j_1, \ldots, j_{q-1} } }{2 \sigma} \leq r \leq \frac{\zeta \lambda^i_{j_1, \ldots, j_{q-1} } }{\sigma} } \left( \frac{ \zeta \lambda^i_{j_1, \ldots, j_{q - 1} } }{r} \right)^{N_{\vec{\jmath}}} f^i_{\vec{\jmath}} (r).
\end{equation}
\end{claim}
\begin{claimproof} Assuming that $(j_1, \ldots, j_{q-1}) \in J_1$ is a nonconstant index, taking $\sigma$ larger if necessary, the claim follows as above from Definition \ref{defn:Nexponents}. Also as above, the claim is trivial in the case that $\mu^i_{j_1, \ldots, j_q} = \xi^{-1} \delta_i,$ so we may always assume that $\lambda^i_{ \vec{\jmath} } \delta_i < \xi \sqrt{ E \left( u_i, \Lambda^i_{\vec{\jmath} } \right) }$ for $\xi$ depending only on $\mathcal{B}.$

Next, assume that the body map ($\vec{\jmath} = (0)$) is a ghost; then (\ref{lowpartial}) holds for $i$ sufficiently large. Fix a nonzero index $j_1$ such that $(j_1) \in J.$ If there is only one such index, then we have $N_{j_1} = 0$ by Definition \ref{defn:Nexponents}.  The LHS of (\ref{puncturedsphere:implication}) is vacuously true, so we may apply Theorem \ref{thm:puncturedsphere} with $n_{j_1} = -2$ to obtain a stronger inequality than (\ref{lowerboundsone:claim2}).

Next, assume that $(j_1)$ is not the only singleton index in $J,$ i.e., there is more than one bubble point. Let $n_{j_1} = N_{j_1}$ and $n_k = - M_{k}$ for $k \neq j_1.$ By Definition \ref{defn:Nexponents}, the hypothesis (\ref{sumniis2}) is satisfied. In view of (\ref{lowerboundsone:Mlowerbound}), we may apply Theorem \ref{thm:puncturedsphere} to obtain the claim for $\vec{\jmath} = (j_1).$

We can now prove the claim for the remaining indices by ``outward'' induction on the bubble tree. Assume that Claim 2 has been proven for all indices up to and including a given index $\vec{\jmath} = (j_1, \ldots, j_q);$ we shall prove it for $(j_1, \ldots, j_q, j_{q + 1}) \in J_0.$ 
Note that from Definition \ref{defn:Nexponents}, we have
\begin{equation}\label{lowerboundsone:Nghostdef}
N_{j_1, \ldots, j_q, j_{q + 1} } = N_{j_1, \ldots, j_q} + \sum_{k\neq j_{q + 1} } M_{j_1, \ldots, j_q, k}.
\end{equation}
Setting $m = -N_{j_1, \ldots, j_q},$ $n_{j_{q + 1}} = N_{j_1, \ldots, j_{q + 1}},$ and $n_k = -M_{j_1, \ldots, j_q, k}$ for $k \neq j_{q + 1},$ in view of the induction hypothesis (\ref{lowerboundsone:claim2}) and the established inequality (\ref{lowerboundsone:Mlowerbound}), Corollary \ref{cor:multiannulus}$b$ implies the claim for $(j_1, \ldots, j_{q + 1}),$ completing the induction.
\end{claimproof}

Claims 1 and 2 imply the hypotheses (\ref{Edf:aassumption}) and (\ref{Edf:bassumption}) of Proposition \ref{prop:Edf}$c,$ which gives the bounds of ($a$); in the statement, we absorb all dependence on $\sigma$ into the constant $\kappa.$ The result of ($b$) follows from Corollary \ref{cor:EdEdbarcomparable}.

\end{proof}

\begin{cor}\label{cor:repulsion} 

Let $\vec{\jmath} = (j_1, \ldots, j_q) \in J,$ and let
$$\cev{\jmath}(\vec{\jmath}) = (j_1, \ldots, j_p)$$
be as in Definition \ref{defn:jandJ}.

\vspace{2mm}

\noindent (a) Suppose that $\cev{\jmath}(\vec{\jmath}) \in J_1,$ i.e. either $\cev{\jmath}(\vec{\jmath}) \neq (0)$ or $\cev{\jmath}(\vec{\jmath}) = (0)$ and the body map is nonconstant. We then have
\begin{equation*}
f^i_{\vec{\jmath}}(r) \geq \kappa \eps \left( \frac{r}{\lambda^i_{\cev{\jmath} (\vec{\jmath})} } \right)^{N_{\vec{\jmath}} + \beta}
\end{equation*}
for $\lambda^i_{j_1, \ldots, j_q} \leq r \leq \zeta \lambda^i_{j_1, \ldots, j_{q - 1}}.$ In particular, if the body map is nonconstant, then we have
\begin{equation}\label{repulsion:universallowerbound}
f^i_{\vec{\jmath}}(r) \geq \kappa \eps r^{N_{\vec{\jmath}} + \beta}
\end{equation}
for all indices $\vec{\jmath} \in J$ and $\lambda^i_{j_1, \ldots, j_q} \leq r \leq \zeta \lambda^i_{j_1, \ldots, j_{q - 1}}.$

\vspace{2mm}

\noindent (b) Let $\vec{k} = (j_1, \ldots, j_q, k_{m + 1}, \ldots, k_n) \in \vec{J}(\vec{\jmath} ).$ We have
\begin{equation*}
f^i_{\vec{\jmath}}(r) \geq \kappa \eps \left( \frac{\lambda^i_{ \vec{k} } }{r} \right)^{M_{\vec{\jmath} } + \beta}.
\end{equation*}

\noindent (c) Let $p \leq m \leq q$ and consider three indices of the form 
$$\vec{\ell} = (j_1, \ldots, j_m),$$
$$\vec{\jmath} = (j_1, \ldots, j_m, j_{m +1}, \ldots, j_q),$$
$$\vec{k} = (j_1, \ldots, j_m, k_{m+ 1}, \ldots, k_n).$$
Suppose that $\vec{k} \in \vec{J}(\vec{\ell})$ 
and $\vec{J} (\vec{\jmath}) \subset \vec{J}(\vec{\ell}),$ i.e., there are no nonconstant indices between $\vec{\ell}$ and $\vec{k}$ or between $\vec{\ell}$ and $\vec{\jmath}.$
For $i$ sufficiently large, we then have
\begin{equation*}
f^i_{j_1, \ldots, j_q}(r) \geq \kappa \eps \frac{ \left( \lambda^i_{\vec{k}} \right)^{M_{ \vec{\ell} } + \beta} r^{N_{\vec{\jmath}} +\beta} }{ \left( \lambda^i_{\vec{\ell}}  \right)^{M_{ \vec{\ell} } + N_{\vec{\jmath} } + 2 \beta} }.
\end{equation*}
\end{cor}
\begin{proof} According to the hypotheses, the relevant indices are connected only by antiholomorphic or flat maps. The estimates then follow by running Theorem \ref{thm:lowerboundsone}$a$ across the connecting indices.
\end{proof}

We can encapsulate the results of this section in the following statement.

\begin{cor} Given $k \in \N,$ 
$\beta > 0,$ and $\eps > 0,$ there exists $\delta > 0$ as follows. Given any $u \in W^{2,2} \left( S^2, S^2 \right)$ with $E(u) \leq 4 \pi k$ and
\begin{equation*}
\| \T(u) \| < \delta,
\end{equation*}
there exists at least one bubble-tree decomposition of $u$ such that the estimates of Theorem \ref{thm:bubbletree} and Corollary \ref{cor:repulsion}$a$-$c$ hold, with $\kappa = \xi = \delta.$
\end{cor}
\begin{proof} This follows by contradiction from the last two results.
\end{proof}

\vspace{10mm}

\section{Main theorems}

The following is our main technical result.

\begin{thm}\label{thm:sequenceloj}
Let $u_i : S^2 \to S^2$ be a bubble-tree convergent almost-harmonic sequence with $E(u_i) \leq 4 \pi k.$ There exist 
constants $C_{\ref{thm:sequenceloj}a-e}$ and $M \geq 1,$ depending on the bubble-tree data $\mathcal{B}$ and $\beta > 0,$ such that for all $i$ sufficiently large, $u_i$ satisfies each of the following statements.

\vspace{2mm}

\noindent (a) Let $\vec{\jmath}$ be a nonzero index and $i \in \N$ sufficiently large. Suppose that there exists 
$\lambda_{j_1, \ldots, j_q} \leq \rho \leq \zeta \lambda_{j_1, \ldots, j_{q-1}}$ such that, letting
\begin{equation*}
W = U^{\rho}_{\rho / 2}(x^i_{\vec{\jmath}} ), \qquad \hat{W} = U^{2\rho}_{\rho/ 4}(x^i_{\vec{\jmath}} ),
\end{equation*}
we have
\begin{equation}\label{sequenceloj:aEassumption} 
E(u_i, \hat{W} )  < \eps_3,
\end{equation}
where $\eps_3$ is the constant of Theorems \ref{thm:preloj}-\ref{thm:localloj}, and
\begin{equation}\label{sequenceloj:aedbarassumption}
M^{-1} \left(  E_{\p}(u_i, W ) - \rho^2 \| \T(u_i) \|_{L^2(\hat{W} )}^2 \right) \leq E_\dbar(u_i, W ) \leq M \left(  E_{\p}(u_i, W ) + \rho^2 \| \T(u_i) \|_{L^2(\hat{W} )}^2 \right).
\end{equation}
 We then have
\begin{equation}\label{sequenceloj:alojest}
\rho ^{-\beta} E(u_i, W ) + \dist \left( E\left( u_i, D_\rho(x^i_{\vec{\jmath}}) \right)  , 4 \pi \Z \right) \leq C_{\ref{thm:sequenceloj}a}  \rho^{1 - \beta} \| \T(u_i) \|.
\end{equation}

\vspace{2mm}

\noindent (b) 
After passing to a subsequence, there exists a subset $J_2 \subset J,$ not containing $(0)$ and with no index preceding any other index in $J_2,$ as well as
$\rho^i_{\vec{\jmath}} \in \LB \lambda^i_{j_1, \ldots, j_q}, \zeta \lambda^i_{j_1, \ldots, j_{q-1}} \RB$ for each $\vec{\jmath} \in J_2$ and $i \in \N,$ such that (\ref{sequenceloj:aEassumption}-\ref{sequenceloj:aedbarassumption}) are satisfied for each $\vec{\jmath} \in J_2$ and $\rho = \rho^i_{\vec{\jmath}},$ and either
\begin{equation*}
E_{\dbar} \left( u_i, S^2 \setminus \cup_{\vec{\jmath} \in J_2} D_{\rho^i_{\vec{\jmath} } } \right) < \eps_3 \quad \text{ or } \quad E_{\p} \left( u_i, S^2 \setminus \cup_{\vec{\jmath} \in J_2} D_{\rho^i_{\vec{\jmath} } } \right) < \eps_3.
\end{equation*}

\vspace{2mm}



\noindent (c) Let $J_2$ and $\left\{\rho^i_{\vec{\jmath}} \mid \vec{\jmath}\in J_2\right\}$ be any collection as in (b). Letting $\lambda_i = \max_{\vec{\jmath} \in J_2} \rho^i_{\vec{\jmath}},$ we have
\begin{equation*}
\dist \left( E(u_i) , 4 \pi \Z \right) \leq C_{\ref{thm:sequenceloj}c} \left( \| \T(u_i) \|^{2} + \lambda_i^{1 - \beta} \| \T(u_i) \| \right).
\end{equation*}

\vspace{2mm}

\noindent (d) Let $L \in \N$ be the integer described in Remark \ref{rmk:L}, and choose $1 < \alpha < \frac{2L+2}{2L + 1}.$ Suppose that the body may $u_\infty$ is nonconstant and holomorphic, 
or that the outermost nonconstant bubbles on each branch are holomorphic. 
Then 
\begin{equation*}
\dist \left( E(u_i) , 4 \pi \Z \right) \leq C_{\ref{thm:sequenceloj}d} \| \T(u_i) \|^{\alpha}.
\end{equation*}
Here $C_{\ref{thm:sequenceloj}d}$ may depend additionally on $\alpha.$

\vspace{2mm}

\noindent (e) Under the same assumption as (d), given any index $\vec{\jmath} \in J_1$ for which the corresponding bubble is nonconstant and antiholomorphic, we have
\begin{equation*}
\left( \lambda^i_{\vec{\jmath}} \right)^{2L + 1 + \beta} \leq C_{\ref{thm:sequenceloj}e} \|\T(u_i) \|.
\end{equation*}
\end{thm}
\begin{proof} We assume that $i$ is sufficiently large that Theorem \ref{thm:lowerboundsone} applies. We then suppress the label $i,$ writing $u = u_i,$ $\Lambda_{\vec{\jmath}} = \Lambda^i_{\vec{\jmath}} ,$ $\delta = \delta_i = \| \T(u_i)\|,$ etc.

Define $J_\p \subset J$ to be the set of all $\vec{\jmath} \in J$ for which
\begin{equation}\label{sequenceloj:EdgeqEdbar}
E_\p \left( u, \hat{\Lambda}_{\vec{\jmath}} \right) \geq E_{\dbar} \left( u, \hat{\Lambda}_{\vec{\jmath}} \right).
\end{equation}
Let $J_\dbar = J_\p^c,$ with $\vec{\jmath} \in J_{\dbar}$ satisfying
\begin{equation}\label{sequenceloj:EdbargeqEd}
E_\p \left( u, \hat{\Lambda}_{\vec{\jmath}} \right) < E_{\dbar} \left( u, \hat{\Lambda}_{\vec{\jmath}} \right).
\end{equation}
We then have
$$J = J_\p \amalg J_\dbar.$$
We first prove ($a$) in the case that $\vec{\jmath} \in J_\p$ and all succeeding indices are also in $J_\p$ (or $\vec{\jmath} \in J_\dbar$ and all succeeding indices are in $J_\dbar$). In this case, for $i$ sufficiently large, we have
$$E_\dbar(u, D_\rho(x_{\vec{\jmath}})) < \eps_3.$$
Then (\ref{sequenceloj:alojest}) follows from the trivial ($\ell = 0$) case of Theorem \ref{thm:localloj}$b.$

We now use ``outward'' induction to establish ($a$) in general. Note that this is a uniformly finite induction, with the number of steps depending only on $\mathcal{B},$ so we can allow the constants to grow at each step. 

Let $\vec{\jmath}$ be an index for which there exists $\rho$ as in (a), and suppose that the claim has been proven for all indices (and $\rho$'s) succeeding $\vec{\jmath}.$ Suppose without loss of generality that $\vec{\jmath} \in J_\p.$ Let $J' \subset J_{\dbar}$ be the set of all antiholomorphic indices immediately following $\vec{\jmath},$ i.e., indices $(j_1, \ldots, j_p, k_{p + 1}, \ldots, k_n ) \in J_\dbar$ such that $(j_1, \ldots, j_p, k_{p + 1}, \ldots, k_m ) \in J_\p$ for all $p < m < n.$ By Theorem \ref{thm:lowerboundsone}$b,$ for each $\vec{\jmath}' \in J'$ there exists $\rho^i_{\vec{\jmath}'}$ such that (\ref{sequenceloj:aedbarassumption}) holds with $\rho = \rho^i_{\vec{\jmath}'}.$ Let
$$\Omega = D_\rho(x_{\vec{\jmath}}) \setminus \cup_{\vec{\jmath}' \in J'} \bar{D}_{\rho_{\vec{\jmath}'}}(x_{\vec{\jmath}'} ).$$
By the energy identity (i.e. the last bullet in Theorem \ref{thm:bubbletree}), for $i$ sufficiently large, we have
$$E_\dbar(u, \Omega) < \eps_3.$$
We may apply Theorem \ref{thm:localloj}$a,$ to obtain
\begin{equation*}
\rho ^{-\beta} E(u, W ) + \dist \left( E\left( u, \Omega \right)  , 4 \pi \Z \right) \leq C_{\ref{thm:localloj}a}  \rho^{1 - \beta} \| \T(u) \|.
\end{equation*}
Meanwhile, by induction, we also have
\begin{equation*}
\dist \left( E\left( u_i, D_{\rho_{\vec{\jmath}'} } \left( x_{\vec{\jmath}'} \right) \right)  , 4 \pi \Z \right) \leq C_{\ref{thm:sequenceloj}a}  \rho_{\vec{\jmath}'}^{1 - \beta} \| \T(u) \|
\end{equation*}
for each $\vec{\jmath}' \in J'.$ Note that
$$ \dist \left( E\left( u, D_{\rho}(x_{\vec{\jmath}}) \right)  , 4 \pi \Z \right) \leq \dist \left( E\left( u, \cup_{\vec{\jmath}' \in J'} D_{\rho_{\vec{\jmath}'} } \left( x_{\vec{\jmath}'} \right) \right), 4 \pi \Z \right) + \dist \left( E\left( u, \Omega \right)  , 4 \pi \Z \right).$$
Combining the previous two estimates, we obtain (\ref{sequenceloj:alojest}) for the given $\rho,$ completing the induction and the proof of ($a$).

To obtain the existence of a collection $J_2$ as in ($b$), suppose without loss of generality that $(0) \in J_\p.$ We can take $J_2$ to be the set of all antiholomorphic indices $\vec{\jmath} \in J_{\dbar}$ immediately following $(0).$ As in the construction of $J'$ above, it is easy to see that this has the required properties.

To obtain ($c$), let $J_2$ be a collection as in ($b$). Set
$$\Omega = S^2 \setminus \cup_{\vec{\jmath} \in J_2} \bar{D}_{\rho_{\vec{\jmath}}}(x_{\vec{\jmath}} ).$$
The global pre-\L ojasiewicz inequality, Theorem \ref{thm:preloj}$a,$ gives
\begin{equation*}
\dist \left( E\left( u, \Omega \right)  , 4 \pi \Z \right) \leq C_{  \ref{thm:preloj} a } \left( \|\T(u) \|^2 + \lambda^{1 - 2 \beta } \| \T(u) \| \right),
\end{equation*}
where $\lambda = \max_{\vec{\jmath} \in J_2} \rho_{\vec{\jmath}}.$ Combining this with (\ref{sequenceloj:alojest}), as above, gives (c).


To obtain ($d$), note that the estimate (\ref{repulsion:universallowerbound}) implies (\ref{preloj:bassumption}) with $m = 2L + 2.$ This allows us to proceed as above, applying Theorems \ref{thm:preloj}$b$-\ref{thm:localloj}$b$ in place of Theorems \ref{thm:preloj}$a$-\ref{thm:localloj}$a.$

Finally, given an index $\vec{\jmath}$ as in ($e$), there must exist $\rho \leq \lambda^i_{\vec{\jmath}}$ as in ($a$), where we may assume that $\rho \leq \rho_{\vec{\jmath}'}$ for some $\vec{\jmath}' \in J_2.$ The claimed estimate now follows from (\ref{preloj:rhoiestimate}).
\end{proof}

\subsection{Proof of Theorems \ref{thm:loj}-\ref{thm:quantitativealphaloj}} Suppose for contradiction that the estimate of Theorem \ref{thm:loj}$a$ fails. We then have a sequence of maps $u_i$ satisfying the assumptions, but for which
\begin{equation}\label{loja:contrad}
\dist(E(u), 4 \pi \Z ) \geq i \| \T(u_i) \|,
\end{equation}
which implies
\begin{equation*}
\| \T(u_i) \| \leq \frac{4 \pi}{i}.
\end{equation*}
In particular, $u_i$ form an almost-harmonic sequence. But we may then pass to a bubble-tree convergent subsequence and apply Theorem \ref{thm:sequenceloj}$c$ to obtain the estimate
\begin{equation*}
\dist(E(u_i), 4 \pi \Z ) \leq C \| \T(u_i) \|
\end{equation*}
for some constant $C$ and $i$ sufficiently large. This contradicts (\ref{loja:contrad}), establishing the result.

Theorem \ref{thm:loj}$b$ follows by contradiction from Theorem \ref{thm:sequenceloj}$d,$ since $\eps$-regularity (Lemma \ref{lemma:epsreg}$a$) ensures that the body map is nonconstant under the assumptions (\ref{loj:Gammahatassumption}-\ref{loj:Gammaassumption}).

To prove Theorem \ref{thm:quantitativelambdaloj}$a,$ one can argue similarly by contradiction to Theorem \ref{thm:sequenceloj}$c;$ in view of  (\ref{quantitativeloj:edbarsmall}-\ref{quantitativeloj:aassumption}) and the estimates of Theorem \ref{thm:bubbletree}, it is possible to refine the given collection $(U_i, \lambda_i)$ to a collection $J_2$ as in Theorem \ref{thm:sequenceloj}$b.$ 
Meanwhile, Theorem \ref{thm:quantitativealphaloj} follows by contradiction from Theorem \ref{thm:sequenceloj}$d,$ whose hypotheses are guaranteed by the assumptions (\ref{quantitativeloj:edbarsmall}) and (\ref{quantitativealphaloj:assumption}).

Finally, Corollary \ref{cor:localonebubbleloj} can be obtained from Theorem \ref{thm:quantitativelambdaloj} as follows. Applying 
Lemma \ref{lemma:comparisonmap} on the ball $B_2,$ we can obtain a map $w : S^2 \to S^2$ with
$$\| \T(w)\|^2 + E(w, B_2 \setminus B_1) \leq C \left( \| \T(u) \|^2 + E(u, B_2 \setminus B_1) \right)$$
which is constant outside $B_2$ and agrees with $u$ inside $B_1.$ On the annulus $U = B_{2 \lambda} \setminus B_\lambda,$ we either have $E_\p(w, U) \geq \frac12 E_\dbar(w, U)$ or $E_\dbar(w, U) \geq \frac12 E_\p(w, U).$ 
In either case, we can apply Theorem \ref{thm:quantitativelambdaloj} with $\ell = 1$ to $w,$ to obtain the desired estimate.
\qed

\vspace{2mm}

\subsection{Proof of Theorem \ref{thm:flowconv} and Corollary \ref{cor:flowconv}}

Let
\begin{equation*}
\delta(t) = \| \T(u(t) ) \|
\end{equation*}
and
$$\Delta(t) = E(u(t)) - 4\pi n,$$
where $n$ is the unique integer such that
$$E(u(t)) \in \left( 4 \pi \left( n - \frac12 \right) , 4\pi \left(n + \frac12 \right) \RB.$$
To avoid confusion with $\delta(t),$ we use $\delta_0$ during the proof in place of the constant $\delta$ in the statement.

By the energy identity for 2D harmonic map flow \cite{dingtian, linwangenergyidentity}, we know that $E(u(t))$ can only jump by an integer multiple of $4 \pi$ at each singular time. Since $|\Delta(t)| < \delta_0 $ for $0 \leq t < T,$ the function $\Delta(t)$ is in fact continuous and decreasing, and the global energy identity
\begin{equation}\label{Deltaglobalenergy}
\Delta(t_2) + \int_{t_1}^{t_2} \delta(t)^2 = \Delta(t_1)
\end{equation}
is valid. 
Let $0 < T_* \leq T$ be the largest time such that the following hold for $0 \leq t < T_*$:
\begin{equation}\label{flowconv:edbarsmall}
E_{\dbar}\left(u(t), \Omega' \right) \leq 2\eps_0,
\end{equation}
\begin{equation}\label{flowconv:edbarUsmall}
\max_i E \left( u(t), \hat{U}_i' \right) \leq 2\eps_0,
\end{equation}
\begin{equation}\label{flowconv:bassumption}
\frac{\kappa}{2} \leq \min_i E_{\p} \left( u(t), U_i' \right),
\end{equation}
and
\begin{equation}\label{flowconv:DeltaT*est}
- 2 \left(  \delta_0^{\frac{\alpha - 2}{\alpha} } + \frac{2 - \alpha}{\alpha}  (T_* - t) \right)^{\frac{\alpha}{\alpha - 2}}  \leq \Delta(t) \leq 2 \left(  \delta_0^{\frac{\alpha - 2}{\alpha} } + \frac{2 - \alpha}{\alpha} t \right)^{\frac{\alpha}{\alpha - 2}}.
\end{equation}
Assuming that $\delta_0$ is sufficiently small, we will show that each of these estimates holds with strict inequality on $\LB 0, T_* \RB.$ This implies that $T_*$ cannot be maximal unless $T_* =  T,$ which will give us the desired estimates.

The following local energy identity is standard (see e.g. \cite{songwaldron} or \cite{toppingannals}):
\begin{equation}\label{flowconv:localenergy}
\left| \int \left( e_\dbar(u(t_2)) - e_\dbar(u(t_1)) \right) \varphi \, dV \right| \leq C \sup |\nabla \varphi | \sqrt{k} \int_{t_1}^{t_2} \delta(t) \, dt,
\end{equation}
where $\varphi$ is compactly supported. The same identity holds with $e_{\p}$ or $e$ in place of $e_\dbar.$ Let
$$U_i' = U^{\frac{5}{2}\lambda_i}_{\frac{4}{5} \lambda_i} (x_i), \quad \hat{U}_i' = U^{3\lambda_i}_{\frac{2}{3} \lambda_i} (x_i), \quad \Omega' = S^2 \setminus \cup \bar{D}_{\frac{4}{5} \lambda_i}(x_i).$$
We can apply (\ref{flowconv:localenergy}) several times, to obtain:
\begin{equation}\label{flowconv:1stlocalenergy}
\begin{split}
E_\dbar(u(t), \Omega') 
& < \eps_0 + \frac{C \sqrt{k}}{\min \lambda_i}  \int_{0}^{t} \delta(t) \, dt,
\end{split}
\end{equation}
\begin{equation}\label{flowconv:2ndlocalenergy}
\begin{split}
E(u(t), \hat{U}_i') 
& < \eps_0 + \frac{C\sqrt{k}}{\min \lambda_i}  \int_{0}^{t} \delta(t) \, dt,
\end{split}
\end{equation}
and
\begin{equation}\label{flowconv:3rdlocalenergy}
\begin{split}
E_\p(u(t), U'_i) 
& \geq \kappa - \frac{C\sqrt{k}}{\min \lambda_i}  \int_{0}^{t} \delta(t) \, dt.
\end{split}
\end{equation}
We first prove (\ref{flowconv:DeltaT*est}) with coefficient $1$ in place of $2$ on both sides. Since the hypotheses of Theorem \ref{thm:quantitativealphaloj} are satisfied, we have
\begin{equation}\label{flowconv:Deltadeltaloj}
|\Delta(t)| \leq \delta(t)^\alpha.
\end{equation}
Here we may ignore the constant by using the inequality for a slightly larger $\alpha$ and taking $\delta_0$ sufficiently small. This gives us
$$\frac{d}{dt} \Delta(t) \leq - \Delta(t)^{\frac{2}{\alpha}}$$
as long as $\Delta(t) \geq 0,$ and
$$\frac{d}{dt} \Delta(t) \leq - (-\Delta(t))^{\frac{2}{\alpha}}$$
thereafter. The inequalities follow from Gronwall's Theorem.

Next, we prove strict inequality in (\ref{flowconv:edbarsmall}-\ref{flowconv:bassumption}). By (\ref{Deltaglobalenergy}), we have 
$$\frac{d}{dt} \Delta(t)^{\frac{\alpha - 1}{\alpha}} = -\frac{\alpha - 1}{\alpha} |\Delta(t)|^{-\frac{1}{\alpha}} \delta^2(t),$$
where we use the convention that $\Delta(t)^{\frac{\alpha - 1}{\alpha}} = - |\Delta(t)|^{\frac{\alpha - 1}{\alpha}}$ if $\Delta(t) < 0.$
Together with (\ref{flowconv:Deltadeltaloj}), this gives
\begin{equation}\label{flowconv:deltaL1est}
\begin{split}
\int_{t_1}^{t_2} \delta(t) \, dt = \int_{t_1}^{t_2} \delta(t)^2 \delta(t)^{-1} \, dt & \leq \int_{t_1}^{t_2} \delta(t)^2 |\Delta(t)|^{-\frac{1}{\alpha}} \, dt \\
& = -\frac{\alpha}{\alpha - 1} \int_{t_1}^{t_2} \frac{d}{dt} \Delta(t)^{\frac{\alpha - 1}{\alpha}} \, dt \\
& \leq \frac{\alpha}{\alpha - 1} \left( \Delta(t_1)^{\frac{\alpha - 1}{\alpha}} - \Delta(t_2)^{\frac{\alpha - 1}{\alpha}} \right).
\end{split}
\end{equation}
Apply (\ref{flowconv:deltaL1est}) with $t_1 = 0$ and $t_2 = T_*,$ to obtain
\begin{equation}\label{flowconv:deltaL1delta0est}
\int_{0}^{T_*} \delta(t) \, dt \leq \frac{2\alpha}{\alpha - 1} \delta_0^{\frac{\alpha - 1}{\alpha} }.
\end{equation}
Provided that $\delta_0$ is sufficiently small, we may conclude from (\ref{flowconv:1stlocalenergy}-\ref{flowconv:3rdlocalenergy}) that strict inequality holds in (\ref{flowconv:edbarsmall}-\ref{flowconv:bassumption}). This implies that in fact $T_* = T,$ so the first estimate of Theorem \ref{thm:flowconv}$a$ now follows from (\ref{flowconv:DeltaT*est}). The second estimate of Theorem \ref{thm:flowconv}$a$ follows from (\ref{flowconv:deltaL1est}) by integrating the harmonic map flow equation in time and applying the $L^2$ norm, then using (\ref{flowconv:deltaL1delta0est}).

To obtain Theorem \ref{thm:flowconv}$b,$ let $u_\infty$ be the body map along any subsequence of times; by (\ref{flowconv:edbarsmall}-\ref{flowconv:bassumption}), this must be nonconstant and holomorphic. The stated $L^2$ estimate follows by letting $t_2 \to \infty$ in (\ref{flowconv:deltaL1est}) and inserting (\ref{flowconv:DeltaT*est}). The $C^k$ improvements follow by combining (\ref{flowconv:deltaL1est}) with standard derivative estimates valid away from the bubbling set $\{z_j\}$ (see e.g. Song-Waldron \cite[\S 3.3]{songwaldron}).

To prove Corollary \ref{cor:flowconv}, suppose that $u_\infty$ is nonconstant, and assume without loss of generality that it is holomorphic. 
Let $\{x_i\} = \{z_i\}$ be the bubbling set along the given sequence $t_i \to \infty,$ put $\lambda_j = \zeta$ for each $j,$ and take $\Omega$ as in Theorem \ref{thm:quantitativealphaloj}-\ref{thm:flowconv}. Since $u(t_i) \to u_\infty$ strongly on $\Omega,$ there exists $\kappa > 0$ such that they hypotheses (\ref{quantitativeloj:edbarsmall}) and (\ref{quantitativealphaloj:assumption}) hold, with $u_0 = u(t_i)$ for $t_i$ sufficiently large. Theorem \ref{thm:flowconv} gives the desired conclusions.

\bibliographystyle{amsinitial}
\bibliography{biblio}

\end{document}